\documentclass[psamsfonts]{amsart}

 \makeatletter
       \def\@makefnmark{%
               \leavevmode
               \raise.9ex\hbox{\check@mathfonts
                       \fontsize\sf@size\z@\normalfont%
                               \@thefnmark}%
       }
       \makeatother
       
\usepackage{amsmath,amssymb,mathrsfs}
\usepackage[dvipdfmx,usenames]{color}
\usepackage[all]{xy}
\usepackage[dvipdfmx]{graphicx}
\usepackage{wrapfig}
\usepackage{tabularx}
\usepackage{supertabular}
\usepackage{setspace}
\usepackage{xypic}
\usepackage{enumerate} 
\usepackage{url}

\usepackage{atbegshi}
\AtBeginShipoutFirst{}

\newtheorem{dfn}{Definition}[section]
\newtheorem{thm}[dfn]{Theorem}
\newtheorem{lem}[dfn]{Lemma}
\newtheorem{prp}[dfn]{Proposition}
\newtheorem{rmk}[dfn]{Remark}
\newtheorem{cor}[dfn]{Corollary}

\newtheorem{mthm}{Main\ Theorem}

\newtheorem{mcor}{Corollary}

\newcommand{\ind}{\mathrm{index}}
\newcommand{\sgn}{\mathrm{sgn}}
\newcommand{\der}{\mathrm{d}}

\title{Toeplitz operators and the Roe-Higson type index theorem}
\author{Tatsuki SETO}

\address{Graduate School of Mathematics, Nagoya University, Furocho, Chikusaku, Nagoya, Japan}
\email{m11034y@math.nagoya-u.ac.jp}

\begin{document}

\begin{abstract}
Let $M$ be a complete Riemannian manifold
and assume that $M$ is partitioned by a hypersurface $N$. 
In this paper we introduce 
a novel class of functions $C_{\mathrm{w}}(M)$ 
on noncompact manifolds, 
which is slightly larger than 
the algebra of Higson functions. 
Out of $\phi$ that belongs to $C_{\mathrm{w}}(M)$ 
we construct an index class $\mathrm{Ind}(\phi , D)$ 
in $K_{1}$-group of the Roe algebra of $M$ 
by using the Kasparov product. 
It is supposed to be a counterpart of 
Roe's odd index class. 
We finally prove that 
Connes' pairing of $\mathrm{Ind}(\phi , D)$ and  
Roe's cyclic $1$-cocycle is equal to the Fredholm index 
of a Toeplitz operator on $N$. 
This is an extension of the Roe-Higson index theorem 
to even-dimensional partitioned manifold. 
\end{abstract}

\maketitle

\section*{Introduction}

Let $M$ be a complete Riemannian manifold 
and assume that $M$ is a partitioned manifold. 
That is,  
there exists a closed submanifold $N$ 
of codimension one such that 
$M$ is decomposed by $N$ into two components $M^{+}$ and $M^{-}$,  
and one has $N = M^{+} \cap M^{-} = \partial M^{+} = \partial M^{-}$; 
see Definition \ref{dfn:partitioned}. 
Let $S \to M$ be a Clifford bundle in the sense of J. Roe 
\cite[Definition 3.4]{MR1670907} 
and $D$ the Dirac operator of $S$. 
Denote by $S_{N}$ the restriction of $S$ to $N$ 
and $\nu$ a unit normal vector field on $N$ pointing from $M^{-}$ into $M^{+}$. 
Then we can equip $S_{N}$ with a $\mathbb{Z}_{2} ( = \mathbb{Z}/2\mathbb{Z})$-graded 
Clifford bundle structure, 
where the $\mathbb{Z}_{2}$-grading on $S_{N}$ 
is defined by using the Clifford action of $\nu$. 
Let $D_{N}$ be the graded Dirac operator on $S_{N}$. 

Let $C^{\ast}(M)$ be the Roe algebra of $M$, 
which is  
a non-unital $C^{\ast}$-algebra 
introduced by Roe in 
\cite{MR1670907}. 
In 
\cite{MR1670907}, 
Roe also defined the odd index class odd-ind$(D) = [u_{D}] \in K_{1}(C^{\ast}(M))$,  
where $u_{D}$ is the Cayley transform of $D$. 

Roe also defined the cyclic $1$-cocycle $\zeta$ 
on a dense subalgebra $\mathscr{X}$ of $C^{\ast}(M)$, 
which is called the Roe cocycle. 
The Roe cocycle is 
constructed to be 
the Connes-Chern character of the 
Fredholm module 
defined by the partition of $M$ 
which corresponds to the 
Wiener-Hopf extension. 
Thus 
Connes' pairing $\langle x , \zeta \rangle$ of $\zeta$ with 
$x \in K_{1}(C^{\ast}(M))$ 
is equal to the index map 
defined by the Wiener-Hopf extension; 
see \cite[Remark 4.14]{MR3000501} and Section \ref{subsec:ext}. 

The Roe cocycle $\zeta$ is related to 
the Poincar$\acute{\text{e}}$ dual of $N$, $pd(N) \in H^{1}_{c}(M)$; 
see \cite[Section 6.1]{MR1147350}. 
In fact, there uniquely exists 
a coarse cohomology class $\alpha \in HX^{1}(M)$  
such that the character map $HX^{1}(M) \to H^{1}_{c}(M)$ sends $\alpha$ to $pd(N)$. 
Moreover the character map $HX^{1}(M) \to HC^{1}(\mathscr{X})$ 
sends $\alpha$ to $[\zeta]$. 
By this relationship of $\zeta$ and $N$, 
it is expected that 
we could pick up some informations of $N$ by using $\zeta$. 
Indeed, Roe proved Connes' pairing $\langle \text{odd-ind}(D) , \zeta \rangle$ 
is equal to the Fredholm index of $D^{+}_{N}$ up to a certain constant multiple  
\cite{MR1670907}. 
In 
\cite{MR1113688}, 
N. Higson gave a simplified proof of a variation of Roe's theorem, thus 
we call it the Roe-Higson index theorem in this paper. 


On the other hand, 
$\ind (D_{N}^{+})$ is $0$ 
for $N$ is of odd dimension; see, for instance 
\cite[Proposition 11.14]{MR1670907}. 
This implies that the Roe-Higson index $\langle \text{odd-ind}(D) , \zeta \rangle$ 
is trivial when $M$ is of even dimension. 
However, Connes' pairing of $\zeta$ with $x \in K_{1}(C^{\ast}(M))$ 
is non trivial in general. 

In this paper, we shall develop an index theorem 
on even dimensional partitioned manifolds, 
which is analogous to the Roe-Higson index theorem. 
For this purpose, 
we need to replace two ingredients, 
odd-ind$(D)$ and the Dirac operator $D_{N}^{+}$ 
by an index class $\mathrm{Ind}(\phi , D) = [\phi] \hat{\otimes} [D]$ 
and a Toeplitz operator on $N$, respectively. 
In order to define this index class $\mathrm{Ind}(\phi , D)$ on $M$,  
we need to introduce a new class of 
$C^{\ast}$-algebra $C_{\mathrm{w}}(M)$, 
which is larger than the Higson functions on $M$ 
and smaller than the bounded continuous functions on $M$; 
see Definition \ref{dfn:algebra}. 
In fact, we use 
$\phi \in GL_{l}(C_{\mathrm{w}}(M))$ and 
$[D] \in KK^{0}(C_{\mathrm{w}}(M) , C^{\ast}(M))$. 
By using the algebra $C_{\mathrm{w}}(M)$, 
this index class $\mathrm{Ind}(\phi , D)$ can be regarded as a 
counterpart of Roe's odd index; see subsection \ref{sec:rel}. 
It turns out Connes' pairing $\langle \mathrm{Ind}(\phi , D) , \zeta \rangle$ 
is equal to the Fredholm index of a Toeplitz operator on $N$ 
up to a certain constant multiple. 
The precise statement as follows:

\begin{mthm}[see Theorem \ref{thm}]
Let $M$ be a complete Riemannian manifold which is partitioned by $N$ as previously. 
Let $S \to M$ be a graded Clifford bundle with the grading $\epsilon$ and 
denote by $D$ the graded Dirac operator of $S$. 
Take $\phi \in GL_{l}(C_{\mathrm{w}}(M))$; see Definition \ref{dfn:algebra}. 
Then the following formula holds: 
\[ \langle \mathrm{Ind}(\phi , D) , \zeta \rangle = -\frac{1}{8\pi i}\ind (T_{\phi |_{N}}). \] 
\end{mthm}

Applying a topological formula of the Fredholm index for Toeplitz operators 
proved by P. Baum and R. G. Douglas \cite{MR679698}, 
we obtain the following: 

\begin{mcor}[see Corollary \ref{cor:mcor}]
Let $M$ be a partitioned manifold partitioned by $(M^{+}, M^{-}, N)$. 
Denote by $\Pi$ the characteristic function of $M^{+}$. 
Let $S \to M$ be a graded Clifford bundle with the grading $\epsilon$, 
and denote by 
$D$ the graded Dirac operator of $S$. 
We assume $\phi \in C^{\infty}(M ; GL_{l}(\mathbb{C}))$ is bounded with 
bounded gradient and $\phi^{-1}$ is also a bounded function. 
Then one has 
\begin{align*}
& \ind \left( \Pi (D+ \epsilon)^{-1}
\begin{bmatrix} \phi & 0 \\ 0 & 1 \end{bmatrix}
(D+\epsilon ) \Pi : \Pi (L^{2}(S))^{l} \to \Pi (L^{2}(S))^{l} \right)  \\
=& 
\int_{S^{\ast}N}
\pi^{\ast}\mathrm{Td}(TN \otimes \mathbb{C})ch(\mathcal{S}^{+})\pi^{\ast}ch(\phi). 
\end{align*}
\end{mcor}

The idea of the proof is as follows. 
Firstly, 
we calculate the Kasparov product $[\phi] \hat{\otimes} [D]$ 
by using the Cuntz picture of $[D]$. 
Secondly, we calculate $\langle \mathrm{Ind}(\phi , D) , \zeta \rangle$ explicitly 
by using the Hilbert transformation and a homotopy of Fredholm operators 
the case for $M = \mathbb{R} \times N$ and 
$\phi = 1 \otimes \psi$ for $\psi \in C^{\infty}(N ; GL_{l}(\mathbb{C}))$. 
Finally, we reduce the general case to $\mathbb{R} \times N$ 
by applying a similar argument in Higson \cite{MR1113688}. 

Set $M = \mathbb{R} \times N$ and 
assume that $N$ is of odd dimension. 
Let $i : N \ni x \mapsto (0,x) \in \mathbb{R} \times N$ be the 
inclusion map. 
Connes \cite{MR679730,MR775126} defined an element $i! \in KK^{1}(N , M)$. 
In this case, the main theorem is derived from the Roe-Higson index theorem 
by applying $i!$ \textit{formally} as follows. 
%
The Dirac operator $D$ on $M$ 
defines an element $[D] \in KK^{0}(M , \mathrm{pt})$. 
Let us take a function $\phi : M \to GL_{l}(\mathbb{C})$ 
and suppose that $\phi$ determines an 
element 
$[[ \phi ]] \in KK^{1}(M,M)$. 
By the Kasparov product, we have an element 
\[ 
[[ \phi ]] \hat{\otimes} [D] \in KK^{1}(M , \mathrm{pt}) 
\]
and also 
\[ 
i! \hat{\otimes} ( [[\phi ]] \hat{\otimes} [D]) 
	= [[\phi |_{N}]] \hat{\otimes}  [D_{N}] \in KK^{0}(N , \mathbb{C}). 
\]
On the other hand, the Roe-Higson index theorem implies 
$\langle A(x) , \zeta \rangle = q_{\ast}(i!  \hat{\otimes} x)$ for $x \in KK^{1}(M, \mathrm{pt})$, 
where $A : KK^{1}(M, \mathrm{pt}) \to K_{1}(C^{\ast}(M))$ is the assembly map 
and $q_{\ast} : K^{0}(N) \to \mathbb{Z}$ is 
the homomorphism induced by the mapping 
$q$ from $N$ to a point. 
Thus we have 
$\langle A([[\phi ]] \hat{\otimes} [D]) , \zeta \rangle 
= q_{\ast} (i! \hat{\otimes} ([[\phi ]] \hat{\otimes} [D])) = \ind (T_{\phi |_{N}})$, 
which is a statement of the main theorem for $M = \mathbb{R} \times N$. 

This formal argument is correct only if $\phi$ 
is an element in $GL_{l}(C_{0}(M))$ since 
the above $KK$ groups are defined as 
$KK^{1}(M,\mathrm{pt}) = KK^{1}(C_{0}(M),\mathbb{C})$, for instance. 
However, if $\phi$ were chosen as an element in $GL_{l}(C_{0}(M))$, 
$\phi$ should take a constant value outside a compact set of $M$. 
This implies that $\phi |_{N}$ is homotopic to a constant function in $GL_{l}(C(N))$
and thus 
$\ind (T_{\phi |_{N}})$ should vanish. 
Therefore, 
in order to obtain non-trivial index, 
we have to employ a larger algebra than $C_{0}(M)$. 

Higson \cite{Higsonalg} introduced such a $C^{\ast}$-algebra $C_{h}(M)$ that contains $C_{0}(M)$, 
which is now called the Higson algebra. 
It plays an important role in a $K$-homological proof of the Roe-Higson index theorem. 
The Higson algebra is defined as follows: 
$C_{h}(M)$ is the $C^{\ast}$-algebra 
generated by all smooth and bounded functions defined on $M$ 
of which gradient is vanishing at infinity \cite[p.26]{Higsonalg}. 
$C_{h}(M)$ contains $C_{0}(M)$ as an ideal and 
is contained in $C_{\mathrm{w}}(M)$ by definition. 
Given $\psi \in C^{\infty}(N)$, 
we note that $\phi = 1 \otimes \psi$ does not belong to $C_{h}(M)$ in general. 
Thus the Higson algebra is not large enough to prove our main theorem. 
On the other hand, we have $\phi \in C_{\mathrm{w}}(M)$.  
Moreover, $C_{\mathrm{w}}(M)$ is the largest $C^{\ast}$-algebra $A$ for which we can define 
$[D]$ as an element in $KK^{0}(A , C^{\ast}(M))$. 
They are reasons why we introduced the $C^{\ast}$-algebra $C_{\mathrm{w}}(M)$ in our main theorem. 

%
%

This paper contains more general method than that of 
author's previous paper \cite{Setotwodim}, 
which proves the case when two dimension 
by elementary method 
and contains a non-trivial example. 

\setcounter{tocdepth}{1}

%
\section{Preliminaries}
%

\subsection{Partitioned manifolds}

We firstly describe a partitioned manifold, 
which is a main object in our main theorem. 

\begin{dfn}
\label{dfn:partitioned}
Let $M$ be an oriented complete Riemannian manifold. 
Assume that the triple $(M^{+}, M^{-}, N)$ satisfies the following conditions: 
\begin{itemize}
\item $M^{+}$ and $M^{-}$ are submanifolds of $M$ of the same dimension as $M$, 
	$\partial M^{+} \neq \emptyset$ and $\partial M^{-} \neq \emptyset$, 
\item $M = M^{+} \cup M^{-}$, 
\item $N$ is a closed submanifold of $M$ of codimension one, 
\item $N = M^{+} \cap M^{-} = -\partial M^{+} = \partial M^{-}$.
\end{itemize}
Then we call $(M^{+}, M^{-}, N)$ a partition of $M$. 
$M$ is also called a partitioned manifold. 
\end{dfn}

\begin{figure}[h]
\vspace*{-1.5\baselineskip}
 \begin{center}
  \includegraphics[width=60mm]{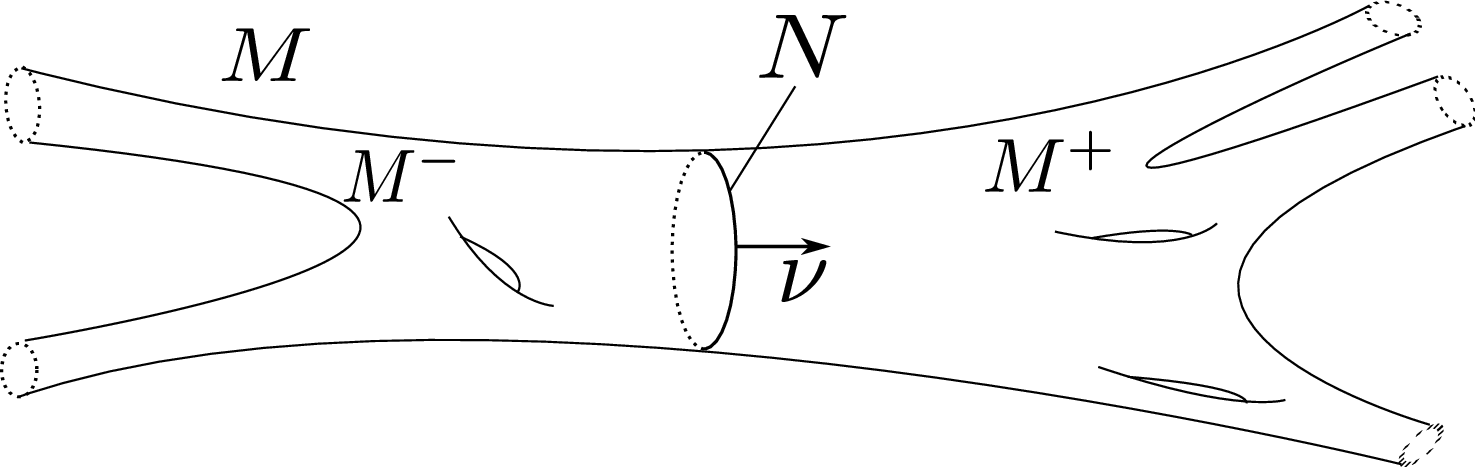}
  \caption{Partitioned manifold}
  \label{fig:integral}
 \end{center}
\end{figure}

For example, we can consider $\mathbb{R} \times N$ is partitioned by 
$(\mathbb{R}_{+} \times N , \mathbb{R}_{-} \times N , \{ 0 \} \times N)$, 
where we set $\mathbb{R}_{+} = \{ t \in \mathbb{R} \,;\, t \geq 0 \}$ 
and $\mathbb{R}_{-} = \{ t \in \mathbb{R} \,;\, t \leq 0 \}$. 

We fix the notation of two functions which are defined by a partition. 

\begin{dfn}
Assume that $M$ is partitioned by $(M^{+}, M^{-}, N)$. 
Denote by $\Pi$ the characteristic function of $M^{+}$ 
and set $\Lambda = 2\Pi - 1$. 
\end{dfn}

\subsection{The Roe algebra}

In this subsection we recall the definition of the Roe algebra $C^{\ast}(M)$ 
and describe some properties of it. 

\begin{dfn}\cite[p.191]{MR996446} 
Let $M$ be a complete Riemannian manifold and 
$S \to M$ a Hermitian vector bundle. 
Denote by $L^{2}(S)$ the $L^{2}$-sections of $S$ and 
$\mathcal{L}(L^{2}(S))$ the bounded operators on $L^{2}(S)$. 
Denote by $\mathscr{X}$ the 
algebra of bounded integral operators on $L^{2}(S)$ 
which have smooth kernels and finite propagation. 
Denote by $C^{\ast}(M)$ the closure of $\mathscr{X}$ 
and call it 
the Roe algebra. 
\end{dfn}

We collect some properties of the Roe algebra which we shall need. 
Let $C_{0}(M)$ be the $C^{\ast}$-algebra of all continuous functions on $M$ vanishing at infinity. 

\begin{prp}\cite{MR1817560,MR996446,MR1147350} 
\label{prp:list}
Let $M$ be a complete Riemannian manifold and $S \to M$ a Hermitian vector bundle. 
Denote by $\mathcal{D}^{\ast}(M)$ 
the $\ast$-subalgebra of $\mathcal{L}(L^{2}(S))$ which contains pseudolocal operators with
finite propagation, where $T \in \mathcal{L}(L^{2}(S))$ is pseudolocal 
if $[f, T] \sim 0$ for all $f \in C_{0}(M)$, 
that is, $[f, T]$ is a compact operator. 
Denote by $D^{\ast}(M)$ the closure of $\mathcal{D}^{\ast}(M)$. 
Then the following holds: 
\begin{enumerate}[$(i)$]
\item $D^{\ast}(M)$ is a unital $C^{\ast}$-algebra. 
\item For all $u \in C^{\ast}(M)$ and $f \in C_{0}(M)$, one has $uf \sim 0$ and $fu \sim 0$. 
\item $C^{\ast}(M)$ is equal to the closure of 
	$\{ u \in \mathcal{L}(L^{2}(S)) \,;\, \text{finite propagation and } uf \sim 0 \text{ and } fu \sim 0 
	\text{ for all } f \in C_{0}(M)  \}$. 
\item $C^{\ast}(M)$ is a closed $\ast$-bisided ideal in $D^{\ast}(M)$. 
\item Let $D$ be a self-adjoint first order elliptic differential operator 
	with finite propagation. 
	Then one has $f(D) \in C^{\ast}(M)$ for all $f \in C_{0}(\mathbb{R})$ and 
	$\chi (D) \in D^{\ast}(M)$ for any chopping function $\chi \in C(\mathbb{R} ; [-1,1])$. 
	Here $\chi \in C(\mathbb{R} ; [-1,1])$ is a chopping function if  
	$\chi$ is an odd function and $\lim_{x \to  \infty} \chi (x) = 1$. 
\end{enumerate}
\end{prp} 

Moreover, we assume $M$ is a partitioned manifold. 
Then we can get the following properties. 

\begin{prp}
\label{prp:loccpt}
If $M$ is a partitioned manifold, then the following holds: 
\begin{enumerate}[$(i)$]
\item For all $u \in C^{\ast}(M)$, one has $[\Pi , u] \sim 0$ and $[\Lambda , u] \sim 0$. 
\item For all $u \in C^{\ast}(M)$ and 
		$\varphi \in C(M)$ satisfies $\varphi = \Pi$ on the complement of 
		a compact set in $M$, one has $[\varphi , u] \sim 0$ . 
\end{enumerate}
\end{prp} 

\begin{proof}
Due to \cite[Lemma 1.5]{MR996446}, 
$[\Pi , u]$ is of trace class for all $u \in \mathscr{X}$. 
So (i) is proved by the definition of $C^{\ast}(M)$. 
Since the support of $\Pi - \varphi$ is compact, 
there exists $f \in C_{0}(M)$ such that $f (\Pi - \varphi) = (\Pi - \varphi)f = \Pi - \varphi$. 
Therefore, we get (ii). 
\end{proof}

\subsection{The Roe cocycle}

We define a certain cyclic $1$-cocycle on $\mathscr{X}$, which is called the Roe cocycle. 

\begin{dfn}
\label{dfn:zeta}
For any $A,B \in \mathscr{X}$, set 
\begin{equation*}
\label{eq:zeta}
\zeta (A,B) = \frac{1}{4}\mathrm{Tr}(\Lambda [\Lambda , A][\Lambda , B]). 
\end{equation*}
We call $\zeta : \mathscr{X} \times \mathscr{X} \to \mathbb{C}$ the Roe cocycle. 
\end{dfn}

\begin{prp}
\cite[Proposition 1.6]{MR996446} 
$\zeta$ is a cyclic $1$-cocycle on $\mathscr{X}$. 
\end{prp}

In our main theorem, we would like to take the pairing of $\zeta$ with 
the index class in $K_{1}(C^{\ast}(M))$. 
For this purpose, we have to extend a domain of $\zeta$. 

\begin{dfn}
\label{dfn:A}
Let $M$ be a partitioned manifold and $S \to M$ a Hermitian vector bundle. 
Then we define a subalgebra $\mathscr{A}$ in $C^{\ast}(M)$ 
such that one has $u \in \mathscr{A}$ if $[\Lambda , u]$ is of trace class. 
\end{dfn}

We note that 
$\mathscr{A}$ is a Banach algebra with norm $\| u \|_{\mathscr{A}} = \| u \| + \| [\Lambda , u] \|_{1}$, 
where $\| \cdot \|$ is the operator norm and $\| \cdot \|_{1}$ is the trace norm. 

\begin{prp}
\label{prp:Ahol}
Let $M$ be a partitioned manifold and $S \to M$ a Hermitian vector bundle. 
Then $\mathscr{A}$ is dense and closed under holomorphic functional calculus in $C^{\ast}(M)$. 
\end{prp} 

\begin{proof}
Since $\mathscr{X} \subset \mathscr{A} \subset C^{\ast}(M)$, $\mathscr{A}$ is dense in $C^{\ast}(M)$. 
By \cite[p.92 Proposition 3]{MR823176} 
and Proposition \ref{prp:loccpt} (ii), 
$\mathscr{A}$ is closed under holomorphic functional calculus in $C^{\ast}(M)$. 
\end{proof}

\begin{rmk}
Due to Definition \ref{dfn:A}, we can extend a domain of $\zeta$ 
to $\mathscr{A}^{+} = \mathscr{A} \oplus \mathbb{C}$. 
\end{rmk}

\subsection{Pairing of the Roe cocycle with an element in $K_{1}$-group}

For any Banach algebra $A$, set $A^{+} = A \oplus \mathbb{C}$ the adjoining a unit. 
Denote by $GL_{l}(A)$ 
the set of invertible elements $u$ in $M_{l}(A^{+})$ such that we have $u - 1 \in M_{l}(A)$. 
Set $K_{1}(A) = \pi_{0}(GL_{\infty}(A))$, the $K_{1}$-group of $A$. 
In this subsection, we describe Connes' pairing of the Roe cocycle with an element in  $K_{1}(C^{\ast}(M))$.

\begin{prp}
\label{prp:induiso}
Let $M$ be a partitioned manifold and $S \to M$ be a Hermitian vector bundle. 
Then the inclusion $i : \mathscr{A} \to C^{\ast}(M)$ induces 
an isomorphism $i_{\ast} : K_{1}(\mathscr{A}) \cong K_{1}(C^{\ast}(M))$. 
\end{prp}

\begin{proof}
Use Proposition \ref{prp:Ahol} and \cite[p.92 Proposition 3]{MR823176}. 
\end{proof}

Due to Proposition \ref{prp:induiso}, 
we can take the pairing of the Roe cocycle with an element in $K_{1}(C^{\ast}(M))$ through the isomorphism 
$i_{\ast} : K_{1}(\mathscr{A}) \cong K_{1}(C^{\ast}(M))$ as follows: 

\begin{dfn}
\cite[p.109]{MR823176} 
Define the map 
\[ \langle \cdot , \zeta \rangle : K_{1}(C^{\ast}(M)) \to \mathbb{C} \]
by $\langle [u] , \zeta \rangle = \frac{1}{8\pi i}\sum_{i,j}\zeta ((u^{-1})_{ji}, u_{ij})$, 
where we assume $[u]$ is represented by an element $u \in GL_{l}(\mathscr{A})$
and $u_{ij}$ is the $(i,j)$-component of $u$. 
We note that this is Connes' pairing of cyclic cohomology with $K$-theory, 
and $1/8\pi i$ is a constant multiple in Connes' pairing. 
\end{dfn}

We can write its pairing by a Fredholm index. 

\begin{prp}
\label{prp:paind}
For any $u \in GL_{l}(C^{\ast}(M))$, one has 
\[ \langle [u] , \zeta \rangle 
	= -\frac{1}{8 \pi i}\ind (\Pi u \Pi : \Pi (L^{2}(S))^{l} \to \Pi (L^{2}(S))^{l}). \]
\end{prp}

\begin{proof}
Since both sides of this equation do not change by homotopy of $u \in GL_{l}(C^{\ast}(M))$, 
it suffices to show the case when $u \in GL_{l}(\mathscr{A})$. 
Then we obtain 
\begin{equation*}
 8 \pi i \langle [u] , \zeta \rangle
= \frac{1}{4}\sum_{i,j}\mathrm{Tr}(\Lambda [\Lambda , (u^{-1})_{ij}][\Lambda , u_{ji}]) 
= \frac{1}{4}\mathrm{Tr}(\Lambda [\Lambda , u^{-1}][\Lambda , u]). 
\end{equation*}

Due to an equality 
$
\Pi - \Pi u^{-1} \Pi u \Pi = - \Pi [\Pi , u^{-1}][\Pi , u]\Pi 
$, 
these two operators 
$\Pi - \Pi u^{-1} \Pi u \Pi$ and $\Pi - \Pi u \Pi u^{-1} \Pi$ 
are of trace class on $\Pi (L^{2}(S))^{l}$. 
Thus we get
\[ \ind (\Pi u \Pi : \Pi (L^{2}(S))^{l} \to \Pi (L^{2}(S))^{l}) 
= \mathrm{Tr}(\Pi - \Pi u^{-1} \Pi u \Pi )
	- \mathrm{Tr}(\Pi - \Pi u \Pi u^{-1} \Pi) \]
by
\cite[p.88]{MR823176}. 
The above calculations implies desired equality. 
\end{proof}

\subsection{Toeplitz operators}

We recall the definition of Toeplitz operators for Dirac operators 
and its index theorem. 
The Fredholm index of the Toeplitz operator appears in our main theorem. 

\begin{dfn}
\label{dfn:Toeplitz}
Let $N$ be a closed Riemannian manifold. 
Let $S_{N} \to N$ be a Clifford bundle 
in the sense of \cite[Definition 3.4]{MR1670907} 
and 
$D_{N}$ the Dirac operator of $S_{N}$. 
Denote by $H_{+}$ the subspace of $L^{2}(S_{N})$ 
generated by non-negative eigenvectors of $D_{N}$ and 
let $P : L^{2}(S_{N}) \to H_{+}$ be the projection. 

Let $\phi \in C(N \,;\, M_{l}(\mathbb{C}))$ be a continuous map 
from $N$ to $M_{l}(\mathbb{C})$. 
Then for any $s \in H_{+}^{l}$, we define the Toeplitz operator 
$T_{\phi} : H_{+}^{l} \to H_{+}^{l}$ by $T_{\phi}s = P(\phi s)$. 
\end{dfn}

Toeplitz operators are Fredholm when 
the range of $\phi$ is contained in the set of invertible matrices. 

\begin{prp}\cite[Lemma 2.10]{MR669904} 
Assume that $\phi$ is a smooth map, then
$[\phi , P]$ is a pseudodifferential operator of order $-1$. 
Therefore $[\phi , P]$ is a compact operator on $L^{2}(S_{N})^{l}$ for all 
$\phi \in C(N \,;\, M_{l}(\mathbb{C}))$. 
This implies a Toeplitz operator $T_{\phi}$ is a Fredholm operator for $\phi \in C(N \,;\, GL_{l}(\mathbb{C}))$. 
\end{prp}

There exists an index theorem for Toeplitz operators. 
We can consider that this index theorem is a corollary of the Atiyah-Singer index theorem. 
Let $\pi : S^{\ast}N \to N$ be the unit sphere bundle of $T^{\ast}N$. 
Denote by $\sigma (x,\xi) \in \mathrm{End}((\pi^{\ast}S_{N})_{(x,\xi )})$ 
the principal symbol of $D_{N}$ 
for all $(x,\xi) \in S^{\ast}N$  
and $\mathcal{S}^{+}_{(x,\xi)}$ the $1$-eigenspace of $\sigma (x,\xi) = i c (\xi)$ . 
Set $\mathcal{S}^{+} = \bigcup_{(x,\xi)} \mathcal{S}_{(x,\xi)}^{+}$, 
then $\mathcal{S}^{+}$ is a subbundle of $\pi^{\ast}S_{N}$. 

\begin{prp}
\label{prp:Toeplitzindex}
\cite[Cororally 24.8]{MR679698}
\cite[Theorem 4]{MR669904} 
The Fredholm index of Toeplitz operators satisfies the following: 
\[ \ind (T_{\phi}) 
	= \langle \pi^{\ast}\mathrm{Td}(TN \otimes \mathbb{C}) 
	ch(\mathcal{S}^{+})\pi^{\ast}ch(\phi) , [S^{\ast}N] \rangle . \] 
\end{prp}

%
\section{Main theorem}
%

\subsection{The index class}
\label{sub:index_class}

In this subsection, we define the odd index class  in $K_{1}(C^{\ast}(M))$. 
After that, we take the pairing of the Roe cocycle with this class. 

Let $(M,g)$ be a complete Riemannian manifold and 
$S \to M$ a graded Clifford bundle with the Clifford action $c$ and the grading $\epsilon$. 
Denote by $D$ the graded Dirac operator of $S$. 
Set $\| f \| = \sup_{x \in M}|f(x)|$ for $f \in C(M)$ and 
$\| X \| = \sup_{x \in M} \sqrt{g_{x}(X,X)}$ for $X \in C^{\infty}(TM)$. 
Denote by $C_{b}(M)$ the $C^{\ast}$-algebra of 
continuous bounded functions on $M$. 

\begin{dfn}
\label{dfn:algebra}
Define $\mathscr{W}(M)$ by the subset of $C^{\infty}(M)$ such that 
one has $f \in \mathscr{W}(M)$ if 
$\| f \| < + \infty$, $\| \mathrm{grad} (f) \| < + \infty$. 
Define $C_{\mathrm{w}}(M)$ by the closure of $\mathscr{W}(M)$ 
by the uniform norm on $M$. 
\end{dfn}

%

Of course, $\mathscr{W}(M)$ is a unital $\ast$-subalgebra of $C_{b}(M)$. 
Therefore, $C_{\mathrm{w}}(M)$ is a unital $C^{\ast}$-algebra.

\begin{rmk}
Let $C_{h}(M)$ be the Higson algebra of $M$, 
that is, $C_{h}(M)$ is the $C^{\ast}$-algebra 
generated by all smooth and bounded functions defined on $M$ 
with which gradient is vanishing at infinity \cite[p.26]{Higsonalg}. 
By definition, one has $C_{h}(M) \subset C_{\mathrm{w}}(M)$. 


We assume $M = \mathbb{R} \times N$ and $\phi \in C^{\infty}(N)$. 
In this case, we have $1 \otimes \phi \in C_{\mathrm{w}}(M)$ 
but $1 \otimes \phi \not \in C_{h}(M)$ in general. 
This is a merit of using $C_{\mathrm{w}}(M)$ (see Section \ref{sec:cyl}). 
\end{rmk}

We define 
a Kasparov $(C_{\mathrm{w}}(M) , C^{\ast}(M))$-module which is made of the Dirac operator $D$. 
We assume that 
the Roe algebra $C^{\ast}(M)$ is 
an even graded $C^{\ast}$-algebra 
and a graded Hilbert $C^{\ast}(M)$-module simultaneously, 
where the grading is induced by $\epsilon$. 
Since $\chi_{0}(x) = x(1+x^{2})^{-1/2}$ is a chopping function, 
the left composition of $F_{D} = D(1+D^{2})^{-1/2} \in D^{\ast}(M)$ 
on an element of $C^{\ast}(M)$ 
is an odd operator on $C^{\ast}(M)$. 

\begin{prp}
Let $\mu : C_{\mathrm{w}}(M) \to \mathbb{B}(C^{\ast}(M))$ be the left composition of the multiplication operator: 
$\mu (f) u = fu \in C^{\ast}(M)$ for $f \in C_{\mathrm{w}}(M)$ and $u \in C^{\ast}(M)$. 
Then one has $[C^{\ast}(M) , \mu , F_{D}] \in KK(C_{\mathrm{w}}(M) , C^{\ast}(M))$. 
\end{prp}

\begin{proof}
Our proof is similar to the Baaj-Julg picture of Kasparov modules 
\cite[Proposition 2.2]{MR715325}. 
Firstly, we obtain $F_{D} \in \mathbb{B}(C^{\ast}(M))$, 
since $F_{D}$ is a self-adjoint bounded operator 
on $L^{2}(S)$ and one has $F_{D}u \in C^{\ast}(M)$ 
for any $u \in C^{\ast}(M)$. 
Because of  
$1 - F_{D}^{2} = 1 - D^{2}(1+D^{2})^{-1} = (1+D^{2})^{-1} \in C^{\ast}(M) = \mathbb{K}(C^{\ast}(M))$ 
and $F_{D}^{\ast} = F_{D}$, 
it suffices to show $[\mu(f), F_{D}] \in C^{\ast}(M)$. 

Now, the following integral formula 
\begin{align*} 
 [ \mu (f) ,  F_{D}] 
&= \frac{1}{\pi}\int_{0}^{\infty}\lambda^{-1/2}
	(1+\lambda )(1 + D^{2} + \lambda )^{-1}[f,D](1 + D^{2} + \lambda )^{-1} d\lambda \\
 & \ \ \ \  + \frac{1}{\pi}\int_{0}^{\infty}\lambda^{-1/2} D(1 + D^{2} + \lambda )^{-1}
			[D,f]D(1 + D^{2} + \lambda )^{-1} d\lambda 				
\end{align*}
is uniformly integrable for any $f \in \mathscr{W}(M)$ 
by $\| (1 + D^{2} + \lambda )^{-1} \| \leq (1+\lambda )^{-1}$ and 
$\| D(1 + D^{2} + \lambda )^{-1} \| \leq (1+\lambda )^{-1/2}$ for any $\lambda \geq 0$, 
and $[\mu (f) , D] = -c (\mathrm{grad} (f)) \in D^{\ast}(M)$ for any $f \in \mathscr{W}(M)$. 
So we obtain $[\mu (f) , F_{D}] \in C^{\ast}(M)$ for any $f \in \mathscr{W}(M)$ 
by $D(1 + D^{2} + \lambda )^{-1} \in C^{\ast}(M)$ for any $\lambda \geq 0$. 
Thus we obtain $[\mu (f) , F_{D}] \in C^{\ast}(M)$ for any $f \in C_{\mathrm{w}}(M)$, 
since we have $\| [\mu (f) , F_{D}] \| \leq 2 \| f \|$ for any $f \in \mathscr{W}(M)$ 
and $\mathscr{W}(M)$ is dense in $C_{\mathrm{w}}(M)$. 
This implies $(C^{\ast}(M) , \mu , F_{D})$ is a Kasparov $(C_{\mathrm{w}}(M) , C^{\ast}(M))$-module. 
\end{proof}

\begin{rmk}
Set $[D] = [C^{\ast}(M), \mu , F_{D}] \in KK(C_{\mathrm{w}}(M) , C^{\ast}(M))$. 
Let $\chi$ be a chopping function. 
Then one has $\chi (D) - F_{D} \in C^{\ast}(M)$ 
by $\chi - \chi_{0} \in C_{0}(M)$. 
Therefore, we obtain $[D] = [C^{\ast}(M), \mu , \chi (D)]$, 
that is, $[D]$ is independent of the choice of a chopping function $\chi$. 
\end{rmk}

Any $\phi \in GL_{l}(C_{\mathrm{w}}(M))$ determines $[\phi ] \in K_{1}(C_{\mathrm{w}}(M))$. 
Due to the Kasparov product 
\[ \hat{\otimes}_{C_{\mathrm{w}}(M)} : K_{1}(C_{\mathrm{w}}(M)) \times KK(C_{\mathrm{w}}(M),C^{\ast}(M)) \to K_{1}(C^{\ast}(M)), \] 
we get the index class in $K_{1}(C^{\ast}(M))$ as follows. 

\begin{dfn}
\label{dfn:indclass}
For any $\phi \in GL_{l}(C_{\mathrm{w}}(M))$, set 
\[ \mathrm{Ind}(\phi , D) = [\phi] \hat{\otimes}_{C_{\mathrm{w}}(M)} [D] \in K_{1}(C^{\ast}(M)). \] 
\end{dfn}

\subsection{The operator on $N$}

Roughly speaking, our main theorem is 
Connes' pairing of the Roe cocycle with $\mathrm{Ind}(\phi , D) \in K_{1}(C^{\ast}(M))$
is calculated by the Fredholm index of a Toeplitz operator on a hypersurface $N$. 
In this subsection, we define its operator. 

Let $M$ be a partitioned manifold 
partitioned by $(M^{+}, M^{-}, N)$. 
Let $S = S^{+} \oplus S^{-}$, $c$ and $D$ are same in Subsection \ref{sub:index_class}. 
Let $\nu \in C^{\infty}(TN)$ be the outward pointing normal unit vector field on $N = \partial M^{-}$. 

Set $S_{N} = S^{+}|_{N}$ and define 
$c_{N} : C^{\infty}(TN) \to C^{\infty}(\mathrm{End}(S_{N}))$ by $c_{N}(X) = c(\nu)c(X)$. 
Then $S_{N}$ can be equipped with a Clifford bundle structure with the Clifford action $c_{N}$. 
Denote by $D_{N}$ the Dirac operator of $S_{N}$. 
We denote the restriction of $\phi \in GL_{l}(C_{\mathrm{w}}(M))$ to $N$ 
by the same letter $\phi$. 
Let $T_{\phi}$ be the Toeplitz operator with symbol $\phi$. 
This Toeplitz operator $T_{\phi}$ is the operator on $N$ in our main theorem.

\subsection{The index theorem}

We recall that we can take Connes' paring of the Roe cocycle with 
$\mathrm{Ind}(\phi , D) \in K_{1}(C^{\ast}(M))$. 
Our main theorem gives the result of its paring. 

\begin{thm}
\label{thm}
Let $M$ be a partitioned manifold partitioned by $(M^{+}, M^{-}, N)$. 
Let $S \to M$ be a graded Clifford bundle with the grading $\epsilon$ 
and denote by $D$ the graded Dirac operator of $S$. 
We denote the restriction of $\phi \in GL_{l}(C_{\mathrm{w}}(M))$ 
to $N$ by the same letter $\phi$. 
Then the following formula holds: 
\[ \langle \mathrm{Ind}(\phi , D) , \zeta \rangle = -\frac{1}{8\pi i}\ind (T_{\phi}). \] 
\end{thm}

If a function $\phi \in C^{\infty}(M;GL_{l}(\mathbb{C}))$ 
satisfies $\| \phi \| < \infty$, 
$\| \mathrm{grad} (\phi ) \| < \infty$ and $\| \phi^{-1} \| < \infty$, 
one has $\phi \in GL_{l}(C_{\mathrm{w}}(M))$ 
since the gradient of $\phi^{-1}$ 
is also bounded. 
The index theorem for Toeplitz operators (see Proposition \ref{prp:Toeplitzindex}) 
implies the following: 

\begin{cor}
\label{cor:mcor}
Let $M$ be a partitioned manifold partitioned by $(M^{+}, M^{-}, N)$,  
and $\Pi$ the characteristic function of $M^{+}$. 
Let $S \to M$ be a graded Clifford bundle with the grading $\epsilon$ and 
denote by $D$ the graded Dirac operator of $S$. 
Assume that $\phi \in C^{\infty}(M ; GL_{l}(\mathbb{C}))$ satisfies $\| \phi \| < \infty$, 
$\| \mathrm{grad} (\phi ) \| < \infty$ and $\| \phi^{-1} \| < \infty$. 
Then one has 
\begin{align*}
& \ind \left( \Pi (D+ \epsilon)^{-1}
\begin{bmatrix} \phi & 0 \\ 0 & 1 \end{bmatrix}
(D+\epsilon ) \Pi : \Pi (L^{2}(S)) \to \Pi (L^{2}(S)) \right)  \\
=& 
\int_{S^{\ast}N} 
\pi^{\ast}\mathrm{Td}(TN \otimes \mathbb{C})ch(\mathcal{S}^{+})\pi^{\ast}ch(\phi). 
\end{align*}
\end{cor}

The proof of Theorem \ref{thm} and Corollary \ref{cor:mcor} 
is provided in Section \ref{sec:cyl} and \ref{sec:gen}.

%
\section{Suspensions and extensions}
%

\subsection{A relationship with Roe's odd index}
\label{sec:rel}

In this subsection, 
we give a formal discussion about  
a relationship with Roe's odd index. 
Firstly, we recall the definition of 
Roe's odd index class $\text{odd-ind}(D)$ \cite[Definition 2.7]{MR996446}. 
Let $M$ be a complete Riemannian manifold, $S \to M$ 
a Clifford bundle, 
$D$ the Dirac operator of $S$ and $\chi$ a chopping function. 
Then we have $\chi (D) \in D^{\ast}(M)$ and 
$q (\chi (D))$ is independent of a choice of $\chi$, 
where $q : D^{\ast}(M) \to D^{\ast}(M)/C^{\ast}(M)$ is a quotient map. 
Moreover, we have $[q ((\chi (D) + 1)/2)] \in K_{0}(D^{\ast}(M)/C^{\ast}(M))$ 
by $\chi^{2} - 1 \in C_{0}(\mathbb{R} ; \mathbb{R})$. 
Let $\delta : K_{0}(D^{\ast}(M)/C^{\ast}(M)) \to K_{1}(C^{\ast}(M))$ be a 
connecting homomorphism of the six-term exact sequence in operator $K$-theory. 
Set $\text{odd-ind}(D) = \delta ([q ((\chi (D) + 1)/2)]) \in K_{1}(C^{\ast}(M))$. 
Remark that we have
$\text{odd-ind}(D) = [(D-i)(D+i)^{-1}]$ if we choose 
$\displaystyle \chi (x) = \frac{1}{\pi}\mathrm{Arg}\left(-\frac{x-i}{x+i}\right)$, 
where we choose the principal value of the argument is $(-\pi , \pi ]$. 
Note that the map defining the odd index class 
is called the assembly map $A : K^{1}(C_{0}(M)) \to K_{1}(C^{\ast}(M))$. 

Secondly, we reconstruct this odd index in terms of $KK$-theory. 
Define $c_{\cdot} : \mathbb{C} \to C_{\mathrm{w}}(M)$ by 
$c_{z}(x) = z$ for $z \in \mathbb{C}$ and $x \in M$. 
Then we have $c_{\cdot} \in KK (\mathbb{C}, C_{\mathrm{w}}(M))$ since this map $c_{\cdot}$ 
is a $\ast$-homomorphism. 
On the other hand, we have 
$[C^{\ast}(M) \oplus C^{\ast}(M), \mu \oplus \mu , \chi (D) \oplus (-\chi (D))] 
\in KK^{1}(C_{\mathrm{w}}(M), C^{\ast}(M))$ 
since 
$\chi_{0}(x) = x(x^{2}+1)^{-1/2}$ is a chopping function 
and we have $\chi - \chi_{0} \in C_{0}(\mathbb{R})$. 
We denote by $[D]$ this $KK$-element. 
Then we obtain $c_{\cdot} \hat{\otimes}_{C_{\mathrm{w}}(M)} [D] = \text{odd-ind}(D)$. 

Finally, we may suppose our Kasparov product is a counterpart
of Roe's odd index as follows. 
We composite the suspension isomorphism 
$KK(\mathbb{C} , C_{\mathrm{w}}(M)) \to KK^{1}(\mathbb{C},  C_{\mathrm{w}}(M) \otimes C_{0}(\mathbb{R}))$ 
and the induced homomorphism by 
an inclusion $C_{\mathrm{w}}(M) \otimes C_{0}(\mathbb{R}) \to C_{\mathrm{w}}(M) \otimes C(S^{1}) \to C_{\mathrm{w}}(M \times S^{1})$. 
Thus we get a homomorphism 
\[ \sigma : KK(\mathbb{C}, C_{\mathrm{w}}(M)) \to KK^{1}(\mathbb{C} , C_{\mathrm{w}}(M \times S^{1})). \] 

On the other hand, there is a homomorphism 
$KK^{1}(C_{\mathrm{w}}(M), C^{\ast}(M)) 
	\to KK^{1}(C_{\mathrm{w}}(M) , \\ C^{\ast}(M \times S^{1}))$
since $KK^{1}$-group is stably isomorphic. 
Let $D_{S^{1}}$ be a Dirac operator on $S^{1}$. 
$D_{S^{1}}$ determines $[D_{S^{1}}] \in KK^{1}(C(S^{1}), \mathbb{C})$. 
By the composition of the Kasparov product $[D_{S^{1}}] \otimes_{\mathbb{C}} -$ and 
the induced map of this $\ast$-homomorphism 
$C_{\mathrm{w}}(M \times S^{1}) \ni f \mapsto f|_{M \times \{ 1 \}} \otimes 1 \in C_{\mathrm{w}}(M) \otimes C(S^{1})$, 
we get a homomorphism 
\[\tau : KK^{1}(C_{\mathrm{w}}(M), C^{\ast}(M))
	\to KK(C_{\mathrm{w}}(M \times S^{1}) , C^{\ast}(M \times S^{1})). \] 
Consequently, by using homomorphisms $\sigma$ and $\tau$, 
we may suppose our Kasparov product is a counterpart of Roe's odd index.

\subsection{Wrong way functoriality}
\label{subsec:i!}

In this subsection, 
we see a correspondence between an index theorem for partitioned manifolds  
with Connes' wrong way functoriality. 
For the simplicity, we assume $M= \mathbb{R} \times N$ with $N$ closed. 
Let $i : \{ \mathrm{pt} \} \to \mathbb{R}$ be an inclusion map 
defined by $i(\mathrm{pt}) = 0$, and $p : \mathbb{R} \to \{ \mathrm{pt} \}$ a constant map. 
Due to Connes (see, for instance, \cite{MR679730,MR775126}), 
they define wrong way functoriality  
$i ! \in KK^{1}(\mathbb{C}, C_{0}(\mathbb{R}))$, 
$(i \times \mathrm{id}_{N})! \in KK^{1}(C(N),C_{0}(M))$ 
and $p! \in KK^{1}(C_{0}(\mathbb{R}),\mathbb{C})$, respectively. 
We note the following: 
\[
i! \otimes_{C_{0}(\mathbb{R})} p! 
= (p \circ i)! 
= 1_{\mathbb{C}} \in KK^{0}(\mathbb{C}, \mathbb{C}). 
\]

Let $D_{N}$ be the Dirac operator on $N$ and 
$D_{\mathbb{R}}$ the Dirac operator on $\mathbb{R}$ 
defined by a spin structure of $\mathbb{R}$. 
These Dirac operators define elements in $K$-homology, 
that is, they define $[D_{N}] \in KK^{\ast}(C(N),\mathbb{C})$ 
and $[D_{\mathbb{R}}] = p! \in KK^{1}(C_{0}(\mathbb{R}),\mathbb{C})$, respectively. 
Moreover, $D_{N}$ and $D_{\mathbb{R}}$ determine the Dirac operator $D_{M}$
on $M = \mathbb{R} \times N$ satisfies 
$[D_{M}] = [D_{\mathbb{R}}] \otimes_{\mathbb{C}} [D_{N}] \in KK^{\ast + 1}(C_{0}(M),\mathbb{C})$. 

Firstly, we assume $\ast = 0$. 
Let $[[E]] \in KK^{0}(C_{0}(M),C_{0}(M))$ 
be a $KK$-element defined by a vector bundle $E \to M$ 
by using the inclusion map $KK^{0}(\mathbb{C},C_{0}(M)) \to KK^{0}(C_{0}(M),C_{0}(M))$.  
Then we have 
\begin{align*}
&\ (i \times \mathrm{id}_{N})! \otimes_{C_{0}(M)} ([[E]] \otimes_{C_{0}(M)} [D_{M}]) \\
=&\ [[E|_{N}]] \otimes_{C(N)} (i \times \mathrm{id}_{N})! 
	\otimes_{C_{0}(M)} ([D_{\mathbb{R}}] \otimes_{\mathbb{C}} [D_{N}] ) \\
=&\ [[E|_{N}]] \otimes_{C(N)} [D_{N}]. 
\end{align*}
Therefore, by using the map 
$q_{\ast} : KK^{0}(C(N),\mathbb{C}) \to KK^{0}(\mathbb{C},\mathbb{C}) \cong \mathbb{Z}$ 
which is the homomorphism 
induced by the mapping $q$ from $N$ to a point, 
we have $q_{\ast}(i! \otimes_{C_{0}(M)} ([[E]] \otimes_{C_{0}(M)} [D_{M}]))$ 
is equal to the Fredholm index of the Dirac operator on $N$ twisted by $E|_{N}$. 
This is a similar formula to the Roe-Higson index theorem. 
Combine the Roe-Higson index theorem, 
this implies the composition of the assembly map $A$ 
with Connes' pairing of $\zeta$ 
is equal to $q_{\ast}(i! \otimes_{C_{0}(M)} -)$. 

On the other hand, we assume $\ast = 1$. 
Take $\phi \in GL_{l}(C_{0}(M))$, 
then it defines an element $[[\phi]] \in KK^{1}(C_{0}(M),C_{0}(M))$ 
by using the inclusion map $KK^{1}(\mathbb{C},C_{0}(M)) \to KK^{1}(C_{0}(M),C_{0}(M))$. 
The similar argument in $\ast = 0$ implies 
\[ 
q_{\ast}(i! \otimes_{C_{0}(M)} ([[\phi ]] \otimes_{C_{0}(M)} [D_{M}])) 
= \ind (T_{\phi |_{N}}).  
\]

However, since $\phi - 1$ vanishes at infinity, 
$\phi|_{N}$ is homotopic to a constant function in $GL_{l}(C(N))$. 
Therefore the right hand side is always $0$.

\subsection{The Roe cocycle and an extension}
\label{subsec:ext}

In this subsection, 
we see a relationship between the Roe cocycle $\zeta$ 
and an extension. 
Let $M$ be a partitioned manifold and 
$\Pi$ the characteristic function of $M^{+}$. 
Set $H = \Pi (L^{2}(S))$. 
Let $q : \mathcal{L}(H) \to \mathcal{Q}(H)$ be the quotient map to the Calkin algebra. 
Define $\sigma : C^{\ast}(M) \to \mathcal{L}(H)$ by $\sigma (A) = \Pi A \Pi$ and 
$\tau : C^{\ast}(M) \to \mathcal{Q}(H)$ by $\tau = q \circ \sigma$. 
Set $E = \{ (A,T) \in C^{\ast}(M) \oplus \mathcal{L}(H) \,;\, \tau (A) = q(T) \}$. 
Then we get an extension $\tau$ of $C^{\ast}(M)$: 
\[  0 \to \mathcal{K}(H) \hookrightarrow E \to C^{\ast}(M) \to 0. \]
This extension $\tau$ corresponds to 
the Fredholm module $(L^{2}(S) , \Lambda)$ 
on $\mathscr{X}$ and 
the Connes-Chern character of 
$(L^{2}(S), \Lambda)$ equals the Roe cocycle. 

By the definition of a pairing 
$\langle \cdot , \cdot \rangle_{\mathrm{ind}} 
: K_{1}(C^{\ast}(M)) \times \mathrm{Ext}(C^{\ast}(M)) \to \mathbb{Z}$ 
and Proposition \ref{prp:paind}, 
we obtain $\langle [u] , \zeta \rangle = \langle [u] , [\tau ] \rangle_{\mathrm{ind}} = \ind (\Pi u \Pi)$ 
up to a certain constant multiple for any $[u] \in K_{1}(C^{\ast}(M))$. 

Moreover, these are equal to 
the connecting homomorphism of this extension: 
$\partial : K_{1}(C^{\ast}(M)) \to K_{0}(\mathcal{K}(H)) \cong \mathbb{Z}$. 
In fact, for any unitary 
$u \in U(C^{\ast}(M)) = \{ u \in U(C^{\ast}(M)^{+}) \,;\, u - 1 \in C^{\ast}(M) \}$, 
denote by $v(u)$ the partial isometry part of the polar decomposition of $\sigma (u)$. 
Then we have $\tau (u) = q (v(u))$ since $\sigma (u)$ is an essential unitary operator on $H$. 
Therefore $(u, v(u)) \in E$ is a partial isometry lift of $u$. 
So we obtain $\partial ([u]) = [\Pi - v(u)^{\ast}v(u)] - [\Pi - v(u)v(u)^{\ast}] 
	\in K_{0}(\mathcal{K}(H))$. 
By the identification $K_{0}(\mathcal{K}(H)) \cong \mathbb{Z}$, 
we have $\partial ([u]) = \ind (v(u)) = \ind (\sigma (u))$. 
Therefore, we obtain 
$\langle [u] , \zeta \rangle = \langle [u] , [\tau ] \rangle_{\mathrm{ind}} 
	= \partial ([u]) = \ind (\Pi u \Pi)$ up to a certain constant multiple.

%
\section{Calculation of the Kasparov product in the index class}
%

\subsection{Explicit formula of the index class}

In this subsection, we represent the index class by an element in $GL_{l}(C^{\ast}(M))$. 
For this purpose, we present $[D]$ by the Cuntz picture of $KK(C_{\mathrm{w}}(M),C^{\ast}(M))$ and then
we calculate Kasparov product $[\phi] \hat{\otimes}_{C_{\mathrm{w}}(M)}[D]$. 
Set 
\[
C^{\ast}_{b}(M) = \left\{ u + \begin{bmatrix} f & 0 \\ 0 & g \end{bmatrix} 
	\,;\, u \in C^{\ast}(M), f,g \in C_{b}(M) \right\}. 
\] 
Then $C^{\ast}_{b}(M)$ is a $C^{\ast}$-subalgebra of $D^{\ast}(M)$ 
and contains $C^{\ast}(M)$ as an essential ideal. 
Let $\chi \in C(\mathbb{R} ; [-1,1])$ be a chopping function.
Set $\eta (x) = (1-\chi (x)^{2})^{1/2} \in C_{0}(\mathbb{R})$. 
Then $\eta$ is a positive even function 
and we have $\eta (D) \in C^{\ast}(M)$. 

\begin{prp}
Let $\iota : C^{\ast}_{b}(M) \to M_{\infty}(C^{\ast}_{b}(M))$ 
be the standard inclusion and $\mathcal{K}$ the $C^{\ast}$-algebra of all compact operators 
on a countably infinite dimensional Hilbert space. 
Set 
$\mathcal{D}_{\chi} = \chi(D) + \epsilon \eta (D) \in D^{\ast}(M)$, 
\[ \psi_{\chi , +}(f) = 
\iota \left( \mathcal{D}_{\chi}\begin{bmatrix} f & 0 \\ 0 & 0 \end{bmatrix}\mathcal{D}_{\chi} \right) 
\text{~and~~} 
\psi_{-}(f) = \iota \left( \begin{bmatrix} 0 & 0 \\ 0 & f \end{bmatrix} \right) . 
\]
Then 
\[ 
(\psi_{+}, \psi_{-}) : 
C_{\mathrm{w}}(M) \to \mathbb{B}(\mathbb{H}_{C^{\ast}(M)}) \triangleright  C^{\ast}(M) \otimes \mathcal{K}
\]  
is a prequasihomomorphism from $C_{\mathrm{w}}(M)$ to $C^{\ast}(M) \otimes \mathcal{K}$ 
in the sense of \cite[Definition 2.1]{MR733641} 
and one has $[D] = [\psi_{+}, \psi_{-}]$ in $KK(C_{\mathrm{w}}(M), C^{\ast}(M))$. 
We note that $(\psi_{+}, \psi_{-})$ is a quasi-homomorphism 
in the sense of \cite[Definition 3.3.1]{MR2464268}. 
Here, we omit the subscript $\chi$ for the simplicity. 
\end{prp}

\begin{proof}
We assume $C^{\ast}(M)^{\mathrm{op}}$ is equipped with the interchanged grading of $C^{\ast}(M)$. 
Then $(C^{\ast}(M)^{\mathrm{op}} , 0 , \chi (D))$ is a degenerate Kasparov $(C_{\mathrm{w}}(M), C^{\ast}(M))$-module. 
So we obtain 
\[ 
[D] = \left[ C^{\ast}(M) \oplus C^{\ast}(M)^{\mathrm{op}} , 
\begin{bmatrix} \mu & 0 \\ 0 & 0 \end{bmatrix}, 
\begin{bmatrix} \chi (D) & 0 \\ 0 & \chi (D) \end{bmatrix}
\right] . 
\]

Since the following  difference  
\begin{equation*} 
\begin{bmatrix} \chi (D) & \epsilon \eta (D) \\ \epsilon \eta (D) & \chi (D) \end{bmatrix}
- 
\begin{bmatrix} \chi (D) & 0 \\ 0 & \chi (D) \end{bmatrix} 
= 
\begin{bmatrix} 0 & \epsilon \eta (D) \\ \epsilon \eta (D) & 0 \end{bmatrix} 
\in M_{2}(C^{\ast}(M)) 
\end{equation*}
is an $C^{\ast}(M)$-compact operator, 
we obtain $[D] = [C^{\ast}(M) \oplus C^{\ast}(M)^{\mathrm{op}} , \mu \oplus 0 , G]$, 
where $G$ is the operator of the first term of the above difference. 
%

The even grading of $C^{\ast}(M)$ is defined by 
the decomposition of $S^{+} \oplus S^{-}$, so we have 
\[
[D] = \left[ E = C^{\ast}(M)^{\mathrm{tri}} \oplus C^{\ast}(M)^{\mathrm{tri}}, 
\begin{bmatrix} \mu & 0 \\ 0 & 0 \end{bmatrix} \oplus 
	\begin{bmatrix} 0 & 0 \\ 0 & \mu \end{bmatrix} , 
\begin{bmatrix} 0 & \mathcal{D} \\ \mathcal{D} & 0 \end{bmatrix} \right] 
\]
under the canonical isomorphism (see \cite[p.119]{MR1656031}): 
\[ 
	KK^{0}(C_{\mathrm{w}}(M) , C^{\ast}(M)) 
	\cong KK^{0}(C_{\mathrm{w}}(M),C^{\ast}(M)^{\mathrm{tri}}),  
\]
where $^{\mathrm{tri}}$ means the trivially grading. 
Now, we conjugate by $\mathcal{D} \oplus 1 \in \mathbb{B}(E)$. 
Then we obtain 
\[
[D] = \left[ E, 
\begin{bmatrix} \mathcal{D}(\mu \oplus 0)\mathcal{D} & 0 \\ 0 & 0 \oplus \mu  \end{bmatrix}, 
\begin{bmatrix} 0 & 1 \\ 1 & 0 \end{bmatrix} \right] . 
\]

By adding a degenerate $(C_{\mathrm{w}}(M),C^{\ast}(M)^{\mathrm{tri}})$-module 
$\left( \mathbb{H}_{C^{\ast}(M)}\oplus \mathbb{H}_{C^{\ast}(M)} , 0 , 
\begin{bmatrix} 0 & 1 \\ 1 & 0 \end{bmatrix} \right)$, we obtain 
\begin{equation*} 
[D] 
= 
\left[ (C^{\ast}(M)^{\mathrm{tri}} \oplus \mathbb{H}_{C^{\ast}(M)})^{2} , 
\begin{bmatrix} 
(\mathcal{D}(\mu \oplus 0)\mathcal{D}) \oplus 0 & 0 \\ 0 & (0 \oplus \mu ) \oplus 0 \end{bmatrix}, 
\begin{bmatrix} 0 & 1 \\ 1 & 0 \end{bmatrix} \right] , 
\end{equation*}
where $\mathbb{H}_{C^{\ast}(M)}$ is a countably generated Hilbert space over $C^{\ast}(M)^{\mathrm{tri}}$. 
We define a unitary operator 
$W : C^{\ast}(M)^{\mathrm{tri}} \oplus \mathbb{H}_{C^{\ast}(M)} \to \mathbb{H}_{C^{\ast}(M)}$ by 
$W(a_{0}, (a_{i})_{i=1}^{\infty}) = (a_{i})_{i=0}^{\infty}$ and conjugate by $W \oplus W$. 
So we obtain 
\[
[D] 
= 
\left[
\mathbb{H}_{C^{\ast}(M)} \oplus \mathbb{H}_{C^{\ast}(M)}, 
\begin{bmatrix} \psi_{+} & 0 \\ 0 & \psi_{-} \end{bmatrix}, 
\begin{bmatrix} 0 & 1 \\ 1 & 0 \end{bmatrix} 
\right] . 
\]
We can show $\psi_{+}(f) \in M_{\infty}(C^{\ast}_{b}(M))$ by using 
\[
\left[ 
\begin{bmatrix} \psi_{+} & 0 \\ 0 & \psi_{-} \end{bmatrix}, 
\begin{bmatrix} 0 & 1 \\ 1 & 0 \end{bmatrix} 
\right] 
\in \mathbb{K}(\hat{\mathbb{H}}_{C^{\ast}(M)^{\mathrm{tri}}}). 
\] 
Therefore, 
\[ 
(\psi_{+},\psi_{-}) : 
C_{\mathrm{w}}(M) \to \mathbb{B}(\mathbb{H}_{C^{\ast}(M)}) \triangleright  C^{\ast}(M) \otimes \mathcal{K}
\] 
is a prequasihomomorphism from $C_{\mathrm{w}}(M)$ to $C^{\ast}(M) \otimes \mathcal{K}$ 
and we obtain $[D] = [\psi_{+},\psi_{-}]$. 
\end{proof}

\begin{rmk}
By definition, one has 
\[
\mathcal{D}\begin{bmatrix} f & 0 \\ 0 & 0 \end{bmatrix}\mathcal{D}
-
\begin{bmatrix} 0 & 0 \\ 0 & f \end{bmatrix} 
= 
\mathcal{D}
\begin{bmatrix} f\eta (D)^{+} & [f , \chi (D)^{-}] \\ 0 & \eta (D)^{-}f \end{bmatrix}
\in C^{\ast}(M) 
\]
for any $f \in C_{\mathrm{w}}(M)$. 
We get another proof of $\psi_{+}(f) \in M_{\infty}(C^{\ast}_{b}(M))$. 
\end{rmk}

The Cuntz picture of Kasparov modules  
suits the Kasparov product with an element in $K_{1}$-group \cite[Remark 1, Theorem 3.3]{MR733641}. 
See also \cite[p.60]{MR2464268}, which contains an explicit formula. 

\begin{prp}
For any $\phi \in GL_{l}(C_{\mathrm{w}}(M))$, one has 
\[
\mathrm{Ind}(\phi , D) = \left[ 
\mathcal{D}
\begin{bmatrix} \phi & 0 \\ 0 & 1 \end{bmatrix}
\mathcal{D}
\begin{bmatrix} 1 & 0 \\ 0 & \phi^{-1} \end{bmatrix}
 \right] \in K_{1}(C^{\ast}(M)). 
\]
\end{prp}

\begin{proof}
Firstly, we obtain 
\[ 
\psi_{+}(\phi - 1)+1 = j \left( \mathcal{D}
\begin{bmatrix} \phi & 0 \\ 0 & 1 \end{bmatrix}
\mathcal{D} \right) 
\text{~and~~}
\psi_{-}(\phi - 1)+1 = j \left(
\begin{bmatrix} 1 & 0 \\ 0 & \phi \end{bmatrix} \right) ,
\]
where $j : GL_{l}(C^{\ast}_{b}(M)) \to GL_{\infty}(C^{\ast}_{b}(M))$ 
is the standard inclusion. Thus we get 
\begin{equation*}
\mathrm{Ind}(\phi , D) 
= \left[ 
\mathcal{D}
\begin{bmatrix} \phi & 0 \\ 0 & 1 \end{bmatrix}
\mathcal{D}
\begin{bmatrix} 1 & 0 \\ 0 & \phi^{-1} \end{bmatrix} \right]
\in K_{1}(C^{\ast}(M)). 
\end{equation*}
\end{proof}

The last of this subsection, we back to Connes' pairing in our main theorem. 

\begin{rmk}
\label{rmk:firstind}
By Proposition \ref{prp:paind}, one has 
\[ 
\langle \mathrm{Ind}(\phi , D) , \zeta \rangle 
= -\frac{1}{8 \pi i}\ind \left( \Pi 
\mathcal{D}
\begin{bmatrix} \phi & 0 \\ 0 & 1 \end{bmatrix}
\mathcal{D}
\begin{bmatrix} 1 & 0 \\ 0 & \phi^{-1} \end{bmatrix} 
\Pi : \Pi (L^{2}(S))^{l} \to \Pi (L^{2}(S))^{l} \right) . 
 \]
On the other hand, 
$\Pi u \Pi$ is Fredholm for any $u \in GL_{l}(C^{\ast}_{b}(M))$ 
because of $[f,\Pi] = 0$ for any $f \in C_{b}(M)$. 
This implies 
\begin{equation*}
\langle \mathrm{Ind}(\phi , D) , \zeta \rangle 
= - \frac{1}{8\pi i }
\ind \left(\Pi \mathcal{D}
\begin{bmatrix} \phi & 0 \\ 0 & 1 \end{bmatrix}
\mathcal{D}\Pi\right). 
\end{equation*}
\end{rmk}

In order to use bellow sections, we fix notation here. 
Set  
$u_{\chi , \phi} = \mathcal{D}_{\chi}\begin{bmatrix} \phi & 0 \\ 0 & 1 \end{bmatrix}\mathcal{D}_{\chi}$ 
and 
$v_{\chi , \phi} = u_{\chi , \phi} - \begin{bmatrix} 1 & 0 \\ 0 & \phi \end{bmatrix}$. 
Then we obtain 
$v_{\chi , \phi} = \mathcal{D}_{\chi}\begin{bmatrix} (\phi - 1)\eta (D)^{+} & [\phi , \chi (D)^{-}]
	\\ 0 & \eta (D)^{-} (\phi - 1) \end{bmatrix}$.

\subsection{Another formula in the special case}
\label{subsec:another}

By Remark \ref{rmk:firstind}, our main theorem is 
the coincidence of two Fredholm indices: 
\[
\ind \left(\Pi u_{\chi , \phi}\Pi\right)
= 
\ind (T_{\phi}). 
\]
Both sides of this equation do not change a homotopy of $\phi$. 
Therefore, it suffices to show the case when $\phi \in GL_{l}(\mathscr{W}(M))$. 
In this case, $\phi : M \to GL_{l}(\mathbb{C})$ is a smooth function 
such that $\| \phi \| < \infty$, 
$\| \mathrm{grad} (\phi ) \| < \infty$ and $\| \phi^{-1} \| < \infty$. 
Moreover, we also assume that $\phi$ satisfies $[|D|, \phi] \in \mathcal{L}(L^{2}(S))$. 
This condition is a technical assumption in this subsection. 
For example, if $S \to M$ has bounded geometry and 
all derivatives of $\phi$ are bounded, then $\phi$ 
satisfies this technical assumption. 
Set $\mathscr{W}_{1}(M) = \{ f \in \mathscr{W}(M) ; [|D| , f ] \in \mathcal{L}(L^{2}(S)) \}$. 
In this subsection, we use $\chi_{0}(x) = x(1+x^{2})^{-1/2}$ as a chopping function, 
that is, we use $\mathcal{D} = \mathcal{D}_{\chi_{0}}$. 

In order to prove our main theorem, 
we perturb the operator $\mathcal{D}
\begin{bmatrix} \phi & 0 \\ 0 & 1 \end{bmatrix}
\mathcal{D}$ 
by a homotopy. 
Firstly, for any $t \in [0,1]$, 
set $F_{t} = t + (1-t)(1+D^{2})^{-1/2} \in D^{\ast}(M)$. 
For any $t \in (0,1]$ and $x \in \mathbb{R}$, set 
\[
f_{t}(x) = \frac{1}{t + (1-t)(1+x^{2})^{-1/2}}. 
\]
Then we obtain $f_{t}(D) \in D^{\ast}(M)$ since $f_{t} - 1/t \in C_{0}(M)$.  
Thus 
$F_{t}$ has a bounded inverse 
$F_{t}^{-1} = f_{t}(D)$ for any $t \in (0,1]$. 

Secondly, because of 
\begin{align*}
(D+\epsilon )^{-1}
\begin{bmatrix} f & 0 \\ 0 & 0 \end{bmatrix}
(D + \epsilon ) \sigma
- 
\begin{bmatrix} 0 & 0 \\ 0 & f \end{bmatrix} \sigma
= (D + \epsilon )^{-1}\begin{bmatrix}
f & -c(\mathrm{grad}(f))^{-} \\ 0 & f
\end{bmatrix} \sigma 
\end{align*}
for any $f \in M_{l}(\mathscr{W}(M))$ and $\sigma \in C^{\infty}_{c}(S)$, we obtain 
\[ 
\left\| (D+\epsilon )^{-1}
\begin{bmatrix} f & 0 \\ 0 & 0 \end{bmatrix}
(D + \epsilon ) \sigma \right\|_{L^{2}}
\leq (2\| f \| + \| \mathrm{grad} (f) \| )\| \sigma \|_{L^{2}}. 
 \]
This implies 
\[ 
\rho (f) = (D+\epsilon )^{-1}
\begin{bmatrix} f & 0 \\ 0 & 0 \end{bmatrix}
(D + \epsilon ) \in \mathcal{L}(L^{2}(S))
\] 
since $C^{\infty}_{c}(S)$ is dense in $L^{2}(S)$. 
Moreover, we obtain $\rho (f) \in C^{\ast}_{b}(M)$ 
by $(D+ \epsilon)^{-1} \in C^{\ast}(M)$ and 
\[ 
\begin{bmatrix}
f & -c(\mathrm{grad}(f))^{-} \\ 0 & f
\end{bmatrix} \in D^{\ast}(M). 
\] 

Finally, set 
$\rho_{0}(f) =  \mathcal{D}
\begin{bmatrix} f & 0 \\ 0 & 0 \end{bmatrix}
\mathcal{D}$ 
and 
$\rho_{t}(f) = F_{t}^{-1} \rho (f) F_{t}$ for any 
$t \in (0,1]$ and $f \in \mathscr{W}(M)$. 
Formally, we set $F_{0}^{-1} = (1+D^{2})^{1/2}$. Then we obtain 
$\rho_{t}(f) = F_{t}^{-1} \rho (f) F_{t} \in \mathcal{L}(L^{2}(S))$ for any 
$t \in [0,1]$ and $f \in \mathscr{W}(M)$. 
We note that we have 
\[ 
\rho_{0}(f) = \mathcal{D}
\begin{bmatrix} f & 0 \\ 0 & 0 \end{bmatrix}
\mathcal{D} 
\text{~and~}
\rho_{1}(f) = \rho (f). 
\] 
This family of bounded operator $t \mapsto \rho_{t}(f)$ 
is continuous in $C^{\ast}_{b}(M)$ for $f \in \mathscr{W}_{1}(M)$. 

\begin{prp}
\label{prp:homotopy-Kasp}
For any $t \in [0,1]$ and $f \in M_{l}(\mathscr{W}_{1}(M))$, 
one has 
$\rho_{t}(f) \in M_{l}(C^{\ast}_{b}(M))$. 
Moreover, $[0,1] \ni t \mapsto \rho_{t}(f) \in M_{l}(C^{\ast}_{b}(M)) \subset M_{l}(\mathcal{L}(L^{2}(S)))$ 
is continuous. 
\end{prp}

\begin{proof}
It suffices to show the case when $l=1$. 

Firstly we show $\rho_{t}(f) \in C^{\ast}_{b}(M)$. 
When $t=0,1$, we already proved. We assume $t \in (0,1)$. 
We have 
\begin{align*}
&\rho_{t}(f) - \begin{bmatrix} 0 & 0 \\ 0 & f \end{bmatrix} \\
=& F_{t}^{-1}(D+\epsilon )^{-1}
\begin{bmatrix}
tf + (1-t)f(1+D^{2})^{-1/2} & tc(\mathrm{grad} (f))^{-} + (1-t)[f,D^{-}(1+D^{2})^{-1/2}] \\
0 & tf + (1-t)(1+D^{2})^{-1/2}f
\end{bmatrix}. 
\end{align*}
Because of 
$F_{t}^{-1} \in D^{\ast}(M)$, $(D+\epsilon )^{-1} \in C^{\ast}(M)$ and 
\[
\begin{bmatrix}
tf + (1-t)f(1+D^{2})^{-1/2} & tc(\mathrm{grad} (f))^{-} + (1-t)[f,D^{-}(1+D^{2})^{-1/2}] \\
0 & tf + (1-t)(1+D^{2})^{-1/2}f
\end{bmatrix}  \in D^{\ast}(M), 
\]
we obtain $\rho_{t}(f) \in C^{\ast}_{b}(M)$. 
 
Next we show continuity of $t \mapsto \rho_{t}(f)$. 
$F_{t}^{-1}$, $\rho (f)$ and $F_{t}$ are bounded operators for any $t \in (0,1]$, 
and $[0,1] \ni t \mapsto F_{t} \in \mathcal{L}(L^{2}(M))$ is continuous. 
Thus $t \mapsto \rho_{t}(f)$ is continuous on $(0,1]$. 
The rest of proof is continuity at $t=0$. 
First, we show $\| (D+\epsilon)^{-1}F_{t}^{-1} \| \leq 2$ for any $t \in [0,1]$. 
Set 
\begin{align*}
g_{t}(x) &= \frac{x}{(1+x^{2})(t + (1-t)(1+x^{2})^{-1/2})} 
\text{~and~} \\ 
h_{t}(x) &= \frac{1}{(1+x^{2})(t + (1-t)(1+x^{2})^{-1/2})}. 
\end{align*}
Then we have 
\[ 
|g_{t}(x)| = \frac{1}{t(|x| + 1/|x|) + (1-t)\sqrt{1 + 1/x^{2}}} \leq \frac{1}{2t + 1-t} \leq 1
\]
and 
$|h_{t}(x)| \leq 1$. 
Thus we obtain $\| (D+\epsilon)^{-1}F_{t}^{-1} \| \leq 2$ by  
\[
(D+\epsilon)^{-1}F_{t}^{-1} = D(1+D^{2})^{-1}F_{t}^{-1} + \epsilon (1+D^{2})^{-1}F_{t}^{-1} 
= g_{t}(D) + \epsilon h_{t}(D). 
\]
By using $\| (D+\epsilon)^{-1}F_{t}^{-1} \| \leq 2$, we can prove continuity at $t = 0$. 
For any $t > 0$, a difference $\rho_{t}(f) - \rho_{0}(f)$ equals 
\begin{align*}
& (D+\epsilon )^{-1} F_{t}^{-1}
	\begin{bmatrix}
	tf - tf(1+D^{2})^{-1/2} & tc(\mathrm{grad} (f))^{-} - t[f,D^{-}(1+D^{2})^{-1/2}] \\
	0 & tf - t(1+D^{2})^{-1/2}f
	\end{bmatrix} \\
& \ \ \ \ + \{ (D+\epsilon)^{-1}F_{t}^{-1} - \mathcal{D} \}
\begin{bmatrix}
f(1+D^{2})^{-1/2} & [f,D^{-}(1+D^{2})^{-1/2}] \\
0 & (1+D^{2})^{-1/2}f
\end{bmatrix}.
\end{align*}
The first term converges to $0$ with the operator norm as $t \to 0$. 

We show the second term converges to $0$ with the operator norm as $t \to 0$. 
Due to  
$
\mathcal{D} - (D+\epsilon )F_{t} = t(D + \epsilon)\{ 1 - (1+D^{2})^{-1/2} \}
$, 
the second term is equal to
\[ 
t(D+\epsilon )^{-1}F_{t}^{-1}\{ (1+D^{2})^{-1/2} - 1 \}  (1+D^{2})^{1/2}
\begin{bmatrix}
	f(1+D^{2})^{-1/2} &  [f,D^{-}(1+D^{2})^{-1/2}] \\
	0 & (1+D^{2})^{-1/2}f
	\end{bmatrix}. 
\]
Therefore, if $(1+D^{2})^{1/2}f(1+D^{2})^{-1/2}$ and $(1+D^{2})^{1/2}[f,D(1+D^{2})^{-1/2}]$ 
are bounded,  
the second term converges to $0$ with the operator norm as $t \to 0$. 
We show that $(1+D^{2})^{1/2}f(1+D^{2})^{-1/2}$ and $(1+D^{2})^{1/2}[f,D(1+D^{2})^{-1/2}]$ are bounded. 
By using the following equalities  
\[ (D^{2} + 1)^{1/2}f(D^{2} + 1)^{-1/2} = [(D^{2} + 1)^{1/2}, f](D^{2} + 1)^{-1/2} + f \text{~and} \] 
\[ (D^{2} + 1)^{1/2}[f , D(D^{2} + 1)^{-1/2}] 
	= [(D^{2} + 1)^{1/2} , f]D(D^{2}+1)^{-1/2} + [f,D], \] 
it suffices to show that $[(D^{2} + 1)^{1/2} , f]$ is a bounded operator. 
Because of $\alpha (x) = \sqrt{x^{2} + 1} - |x| \in C_{0}(\mathbb{R})$, 
we have $\alpha (D) \in \mathcal{L}(L^{2}(S))$. 
This implies $[(D^{2} + 1)^{1/2} , f]$ is bounded 
if and only if $[|D| , f]$ is bounded. 
We note that boundness of $[|D| , f]$ 
is required the definition of the algebra $\mathscr{W}_{1}(M)$. 
Hence $(D^{2} + 1)^{1/2}f(D^{2} + 1)^{-1/2}$ and $(D^{2} + 1)^{1/2}[f , D(D^{2} + 1)^{-1/2}]$ are bounded. 
Thus the second term converges to $0$ as $t \to 0$. 
Therefore $t \mapsto \rho_{t}(f)$ is continuous. 
\end{proof}

Due to Proposition \ref{prp:homotopy-Kasp}, the following maps   
\[ 
\Pi \{\rho_{t}(\phi - 1 )+ 1\} \Pi : \Pi (L^{2}(S))^{l} \to \Pi (L^{2}(S))^{l}
\]
determines a continuous family of Fredholm operators for any $\phi \in GL_{l}(\mathscr{W}_{1}(M))$. 
Therefore, we obtain 
\[
\langle \mathrm{Ind}(\phi , D) , \zeta \rangle 
= 
-\frac{1}{8\pi i}
\ind \left(\Pi (D+ \epsilon)^{-1}
\begin{bmatrix} \phi & 0 \\ 0 & 1 \end{bmatrix}
(D+\epsilon )\Pi\right) 
\]
for any $\phi \in GL_{l}(\mathscr{W}_{1}(M))$. 

\begin{rmk}
In the definition of $\rho_{t}$, we do not use the assumption $[|D| , f] \in \mathcal{L}(L^{2}(S))$. 
In particular, one has $\rho (f) \in C^{\ast}_{b}(M)$ for $f \in \mathscr{W}(M)$.  
Set $\varrho (\phi) = \rho (\phi - 1) + 1$ for any $\phi \in GL_{l}(\mathscr{W}(M))$. 
Then the operator $\Pi \varrho (\phi) \Pi$ is Fredholm 
for all $\phi \in GL_{l}(\mathscr{W}(M))$. 
\end{rmk}

%
\section{The case for $\mathbb{R} \times N$}
\label{sec:cyl}
%

Let $N$ be a closed manifold. 
In this section, we prove Theorem \ref{thm} in the case that $M = \mathbb{R} \times N$. 
Recall that $\mathbb{R} \times N$ is partitioned by 
$(\mathbb{R}_{+} \times N , \mathbb{R}_{-} \times N , \{ 0 \} \times N)$. 
Let $S_{N} \to N$ be a Clifford bundle, 
$c_{N}$ the Clifford action on $S_{N}$ and $D_{N}$ the Dirac operator on $S_{N}$. 
Given $\phi \in C^{\infty}(N ; GL_{l}(\mathbb{C}))$, 
we define the map $\tilde{\phi} : \mathbb{R} \times N \to GL_{l}(\mathbb{C})$ 
by $\tilde{\phi} (t,x) = \phi (x)$. 
We often denote $\tilde{\phi}$ by $\phi$ in the sequel. 
Note that we have $\phi \in GL_{l}(\mathscr{W}_{1}(\mathbb{R} \times N))$. 

Let $p : \mathbb{R} \times N \to N$ be the projection to $N$. 
Set $S = p^{\ast}S_{N} \oplus p^{\ast}S_{N}$ and 
$\epsilon = 1 \oplus (-1)$, where $\epsilon$ is the grading operator on $S$. 
Then we define a Clifford action $c : C^{\infty}(TM) \to C^{\infty}(\mathrm{End}(S))$ by
\[
c(\der / \der t) = 
\begin{bmatrix}
0 & 1 \\ -1 & 0
\end{bmatrix}, \ 
c(X) = 
\begin{bmatrix}
0 & c_{N}(X) \\ c_{N}(X) & 0
\end{bmatrix} \text{~for~all~}X \in C^{\infty}(TN).   
\]
Here $\der / \der t$ is a coordinate unit vector field on $\mathbb{R}$. 
Then $S \to M$ is a Clifford bundle 
and the Dirac operator $D$ of $S$ is given by 
\[
D = 
\begin{bmatrix}
0 & \der / \der t + D_{N} \ \\ -\der / \der t + D_{N} & 0
\end{bmatrix}. 
\]
Denote by $H_{+}$ the subspace of $L^{2}(S_{N})$ 
which is generated by non-negative eigenvectors of $D_{N}$. 
Also denote by $H_{-}$ the orthogonal complement of $H_{+}$ in $L^{2}(S)$. 
Set $F = 2P-1$, where $P$ is the projection to $H_{+}$. 

Due to Subsection \ref{subsec:another}, 
it suffices to show 
\[
\ind \left(\Pi (D+ \epsilon)^{-1}
\begin{bmatrix} \phi & 0 \\ 0 & 1 \end{bmatrix}
(D+\epsilon )\Pi\right) 
= 
\ind (T_{\phi}). 
\]
For this purpose, we perturb the operator $\Pi \varrho (\phi) \Pi$ by a homotopy. 
We firstly estimate the supuremum of some functions to prove a continuity of the homotopy.  

\begin{lem}
\label{lem:sup}
\begin{enumerate}[$(i)$]
\item Set 
	\[ f_{s}(x) = \frac{x}{x^{2} + (1-s)^{2}} \text{ and } 
		g_{s}(x) = \frac{1}{x^{2} + (1-s)^{2}}  \]
	for all $s \in [0,1]$ and $x \in \mathbb{R} \setminus (-s, s)$. 
	Then one has $\sup_{x} |f_{s}(x)| \leq 2$ and $\sup_{x} |g_{s}(x)| \leq 2$ for all $s \in [0,1]$. 
\item Set
	\[ \mu_{\lambda , s}(x) = \frac{1}{x^{2} + \{ (1-s)\lambda + s\sgn (\lambda ) \}^{2} + (1-s)^{2}}	\]
	and $\nu_{\lambda , s}(x) = x\mu_{\lambda , s}(x)$ 
	for all $\lambda \in \mathbb{R}$, $s \in [0,1)$ and $x \in \mathbb{R}$, 
	where $\sgn (\lambda)$ is $1$ if $\lambda \geq 0$ or $-1$ if $\lambda < 0$. 
	Then one has 
	\[ \sup_{x} |\mu_{\lambda , s}(x)| \leq \frac{1}{(1-s)^{2}(\lambda^{2} + 1)}
	\text{ and }
	\sup_{x} |\nu_{\lambda , s}(x)| \leq \frac{1}{2(1-s)\sqrt{\lambda^{2} + 1}} \]
	for all $\lambda \in \mathbb{R}$, $s \in [0,1)$. 
\end{enumerate}
\end{lem}

\begin{proof}

\noindent
(i) 
For $0 \leq s \leq 1/2$, we have $|f_{s}(x)| \leq f_{s}(1-s) \leq 1$. 
For $1/2 \leq s \leq 1$, we have $|f_{s}(x)| \leq f_{s}(s) \leq 2$. 
This implies $\sup_{x} |f_{s}(x)| \leq 2$. 
On the other hand, we have $|g_{s}(x)| \leq g_{s}(s) \leq 2$. 

\noindent
(ii) 
For $\lambda \geq 0$, we have $(1-s)\lambda + s \sgn (\lambda) \geq (1-s)\lambda \geq 0$. 
On the other hand, for $\lambda < 0$, we have $(1-s)\lambda + s \sgn (\lambda) \leq (1-s)\lambda < 0$. 
So we obtain $|\mu_{\lambda , s}(x)| \leq h_{\lambda , s}(0) \leq 1/(1-s)^{2}(\lambda^{2} + 1)$. 

On the other hand, we obtain 
\[ |\nu_{\lambda , s}(x)| \leq \nu_{\lambda , s}
	\left(\sqrt{\{ (1-s)\lambda + s\sgn (\lambda ) \}^{2} + (1-s)^{2}}\right)
	\leq \frac{1}{2(1-s)\sqrt{\lambda^{2} + 1}}. 
\]
\end{proof}

\begin{prp}
\label{prp:homotopy}
Set 
\[
D_{s} = 
\begin{bmatrix}
0 & \der / \der t + (1-s)D_{N} + sF \\ -\der / \der t + (1-s)D_{N} + sF & 0
\end{bmatrix} 
\]
for all $s \in [0,1]$ 
and 
\[
u_{\phi, s} = (D_{s} + (1-s)\epsilon )^{-1}
\begin{bmatrix} \phi & 0 \\ 0 & 1 \end{bmatrix}
(D_{s} + (1-s)\epsilon ). 
\]
Then the map $[0,1] \ni s \mapsto u_{\phi , s} \in \mathcal{L}(L^{2}(S)^{l})$ is continuous. 
\end{prp}

\begin{proof}

It suffices to show the case when $l=1$. 
Since we have $(\der / \der t)^{\ast} = -\der / \der t$ and $D_{N}$ is a Dirac operator on $N$, 
$D_{s}$ is a self-adjoint closed operator densely defined on $\mathrm{domain} (D_{s}) = \mathrm{domain} (D)$. 

Next we show $\sigma (D_{s}) \cap (-s, s) = \emptyset$ for all $s \in (0,1]$. 
Set 
\[
T_{s} = 
\begin{bmatrix}
0 & \der / \der t + (1-s)D_{N} \\ -\der / \der t + (1-s)D_{N} & 0
\end{bmatrix} \text{ and } 
J = 
\begin{bmatrix}
0 & F \\ F & 0
\end{bmatrix}. 
\]
These operators $T_{s}$ and $J$ are self-adjoint and we have $D_{s} = T_{s} + sJ$ 
and $T_{s}J+JT_{s} = 2(1-s)D_{N}F \geq 0$ on $\mathrm{domain} (D)$. 
So for any $\sigma \in \mathrm{domain} (D)$, we obtain 
\[
\| D_{s}\sigma \|_{L^{2}}^{2} 
= \|T_{s}\sigma \|_{L^{2}}^{2} + s^{2}\|J\sigma \|_{L^{2}}^{2} 
	+ s \langle (T_{s}J+JT_{s})\sigma , \sigma \rangle_{L^{2}}
\geq s^{2}\| J\sigma \|_{L^{2}}^{2} = s^{2}\| \sigma \|_{L^{2}}^{2}. 
\]
This implies $\sigma (D_{s}) \cap (-s, s) \neq \emptyset$. 
In particular, $D_{1}$ has a bounded inverse. 

On the other hand, when $s \in [0, 1)$, 
we have $(D_{s} + (1-s)\epsilon )^{-1} \in \mathcal{L}(L^{2}(S))$ since
$(D_{s} + (1-s)\epsilon )^{2} = D_{s}^{2} + (1-s)^{2}$ is invertible. 
Therefore $u_{\phi , s}$ is well defined as a closed operator on $L^{2}(S)$ 
with $\mathrm{domain} (u_{\phi, s}) = \mathrm{domain} (D)$ for all $s \in [0,1]$. 
Thus we obtain $u_{\phi , s} \in \mathcal{L}(L^{2}(S))$ by 
\[ 
u_{\phi , s} = \begin{bmatrix} 1 & 0 \\ 0 & \phi \end{bmatrix} + (D_{s} + (1-s)\epsilon )^{-1}
\begin{bmatrix}
(1-s)(\phi - 1) & -(1-s)c_{N}(\mathrm{grad} (\phi ) ) + s[\phi , F] \\ 0 & (1-s)(\phi - 1) 
\end{bmatrix}. 
\]

Next we show continuity of $[0,1] \ni s \mapsto u_{\phi , s} \in \mathcal{L}(L^{2}(S))$. 
First, because of 
$ 
(D_{s} + (1-s)\epsilon )^{-1} = f_{s}(D_{s}) + (1-s)\epsilon g_{s}(D_{s})
$,   
we have 
\begin{equation} 
\label{eq:bdd}
\| (D_{s} + (1-s)\epsilon )^{-1} \| \leq \sup_{x} |f_{s}(x)| + (1-s)\sup_{x} |g_{s}(x)| \leq 4 \tag{$\ast$}
\end{equation}
by Lemma \ref{lem:sup}. 
Therefore $\{ \| (D_{s} + (1-s)\epsilon )^{-1} \| \}_{s \in [0,1]}$ is a bounded set. 

Next, for any $s,s' \in [0,1]$, 
a difference $u_{\phi , s} - u_{\phi , s'}$ equals 
\begin{align*}
& (D_{s} + (1-s)\epsilon )^{-1}
\begin{bmatrix}
(s'-s)(\phi - 1) & (s-s')c_{N}(\mathrm{grad} (\phi ) ) + (s-s')[\phi , F] \\ 0 & (s'-s)(\phi - 1) 
\end{bmatrix} \\
&~~~~~+ 
\{ (D_{s} + (1-s)\epsilon )^{-1} - (D_{s'} + (1-s')\epsilon )^{-1}\}
\begin{bmatrix}
(1-s')(\phi - 1) & -(1-s')c_{N}(\mathrm{grad} (\phi ) ) + s'[\phi , F] \\ 0 & (1-s')(\phi - 1) 
\end{bmatrix} \\
=:& ~ \alpha_{s,s'} + \beta_{s,s'}.
\end{align*}
The first term $\alpha_{s,s'}$ converges to 0 with the operator norm as $s \to s'$. 

The rest of proof is the second term $\beta_{s,s'}$ converges to 0. 
Firstly, we assume $s' = 1$. Then we obtain 
\begin{equation*}
\beta_{s,1} 
= \{ (D_{s} + (1-s)\epsilon )^{-1} - D_{1}^{-1} \} 
\begin{bmatrix}
0 & [\phi , F] \\ 0 & 0 
\end{bmatrix} 
\end{equation*}
and 
\begin{align*}
& (D_{s} + (1-s)\epsilon )^{-1} - D_{1}^{-1} \\
=& ~ (s-1)(D_{s} + (1-s)\epsilon )^{-1}D_{1}^{-1}
\begin{bmatrix}
0 & D_{N} \\ D_{N} & 0
\end{bmatrix}
 + 
(1-s)(D_{s} + (1-s)\epsilon )^{-1}(J - \epsilon )D_{1}^{-1} 
\end{align*}
since $D_{N}$ commutes $F$ and $\der / \der t$ on $\mathrm{domain} (D)$, respectively. 
Therefore, we have 
\begin{align*}
\beta_{s,1} 
=& ~ (s-1)(D_{s} + (1-s)\epsilon )^{-1}D_{1}^{-1}
\begin{bmatrix}
0 & 0 \\ 0 & D_{N}[\phi , F]
\end{bmatrix} \\
 &~~~~~~~~~~~~~~~~+ 
(1-s)(D_{s} + (1-s)\epsilon )^{-1}(J - \epsilon )D_{1}^{-1}
\begin{bmatrix}
0 & [\phi , F] \\ 0 & 0 
\end{bmatrix} 
\end{align*} 
and thus $\beta_{s,1}$ 
converges to 0 with the operator norm as $s \to 1$ since $D_{N}[\phi , F]$ 
is a pseudo-differential operator of order $0$ on $N$ and $\| (D_{s} + (1-s)\epsilon )^{-1}\|$, 
$\|J\|$, $\|\epsilon\|$ and $\|D_{1}^{-1}\|$ 
are uniformly bounded. 

We assume $0 \leq s' < 1$. Since an operator  
\[
\begin{bmatrix}
(1-s')(\phi - 1) & -(1-s')c_{N}(\mathrm{grad} (\phi ) ) + s'[\phi , F] \\ 0 & (1-s')(\phi - 1) 
\end{bmatrix}
\]
is bounded, it suffices to show 
\[ \| (D_{s} + (1-s)\epsilon )^{-1} - (D_{s'} + (1-s')\epsilon )^{-1} \| \to 0 \]
as $s \to s'$. 
We have 
\begin{align*} 
& (D_{s} + (1-s)\epsilon )^{-1} - (D_{s'} + (1-s')\epsilon )^{-1} \\
=& (s-s')(D_{s} + (1-s)\epsilon )^{-1}
	\begin{bmatrix}
	0 & D_{N} \\ D_{N} & 0 
	\end{bmatrix}
	(D_{s'} + (1-s')\epsilon )^{-1} \\ 
	&~~~~~~~~~~~+ 
	(s'-s)(D_{s} + (1-s)\epsilon )^{-1}(J - \epsilon )(D_{s'} + (1-s')\epsilon )^{-1}
\end{align*}
and the second term converges to $0$ with the operator norm as $s \to s'$ by (\ref{eq:bdd}). 
So it suffices to show that an operator $U$ defined by 
\begin{align*}
U
&=  \begin{bmatrix}
	0 & D_{N} \\ D_{N} & 0 
	\end{bmatrix}
	(D_{s'} + (1-s')\epsilon )^{-1} \\
&= \begin{bmatrix}
	D_{N}A_{s'}^{-1}(\der /\der t + (1-s')D_{N} + s'F) & -(1-s')D_{N}A_{s'}^{-1} \\
	(1-s')D_{N}A_{s'}^{-1} & D_{N}A_{s'}^{-1}(-\der /\der t + (1-s')D_{N} + s'F)
	\end{bmatrix}
\end{align*}
is a bounded operator on $L^{2}(S) = L^{2}(\mathbb{R})^{2} \otimes L^{2}(S_{N})$, where set 
\[
A_{s'} = -\der^{2} / \der t^{2} + \{ (1-s')D_{N} + s'F \}^{2} + (1-s')^{2}. 
\] 
Now, if $D_{N}A_{s'}^{-1}$, $iD_{N}A_{s'}^{-1}\der / \der t$ and $D_{N}A_{s'}^{-1}D_{N}$ 
are bounded, then $U$ is also bounded. 
We show $D_{N}A_{s'}^{-1}D_{N}$ is bounded. 
Denote by $E_{\lambda}$ the $\lambda$-eigenspace of $D_{N}$. 
Then $D_{N}A_{s'}^{-1}D_{N}$ acts as 
$
\lambda^{2} 
	\{ -\der^{2} / \der t^{2} + \left( (1-s')\lambda + s' \sgn (\lambda ) \right)^{2} + (1-s')^{2} \}^{-1} 
$
on $L^{2}(\mathbb{R}) \otimes E_{\lambda}$. 
This operator equals to $\lambda^{2} \mu_{\lambda , s'}(i\der / \der t)$ and 
we have $\| \lambda^{2} \mu_{\lambda , s'}(i\der / \der t) \| \leq 1/(1-s')^{2}$ 
by Lemma \ref{lem:sup}. 
Therefore we obtain $\| D_{N}A_{s'}^{-1}D_{N}\| \leq 1/(1-s')^{2}$. 
Similarly, we can show $\| D_{N}A_{s'}^{-1}\| \leq 1/(1-s')^{2}$ (use $\mu_{\lambda , s'}$) and 
$\| iD_{N}A_{s'}^{-1}\der / \der t\| \leq 1/2(1-s')$ (use $\nu_{\lambda , s'}$). 
Thus $U$ is bounded. 
Therefore we obtain 
\[ \| (D_{s} + (1-s)\epsilon )^{-1} - (D_{s'} + (1-s')\epsilon )^{-1} \| \to 0 \]
as $s \to s'$ as required. 

\end{proof}

By Proposition \ref{prp:homotopy}, 
$\Pi u_{\phi , s} \Pi$ is a continuous path in $\mathcal{L}(\Pi (L^{2}(S))^{l})$. 
In fact, this continuous path is a desired homotopy of Fredholm operators. 

\begin{prp}
\label{prp:inRoe}
Set 
\[ 
v_{\phi , s} = 
u_{\phi , s} - \begin{bmatrix} 1 & 0 \\ 0 & \phi \end{bmatrix} 
\]
for all $s \in [0,1]$. One has $[\Pi , v_{\phi , s}]  \sim 0$. 
Therefore 
$
\Pi u_{\phi, s} \Pi : \Pi (L^{2}(S)) \to \Pi (L^{2}(S))
$
is a Fredholm operator. 
\end{prp}

\begin{proof}
It suffices to show the case when $l=1$. 
Due to Proposition \ref{prp:homotopy} and closedness of $\mathcal{K}(L^{2}(S))$, 
we may assume $s \in [0,1 )$. 

First, we show $g (D_{s} + (1-s)\epsilon )^{-1} \sim 0$ 
for any $g \in C_{0}(\mathbb{R})$. 
Since $C_{c}^{\infty}(\mathbb{R})$ is dense in $C_{0}(\mathbb{R})$, 
it suffices to show the case when $g \in C_{c}^{\infty}(\mathbb{R})$. 
Because $T_{s}$ (see in the proof of Proposition \ref{prp:homotopy}) 
is a first order elliptic differential operator and 
$g$ commutes with a operator on $N$, 
we have 
\[
\| g (D_{s} + (1-s)\epsilon )^{-1}u \|_{H^{1}} 
\leq C ( \| g(D_{s} + (1-s)\epsilon )^{-1}u \|_{L^{2}} + \| T_{s}g(D_{s} + (1-s)\epsilon )^{-1} u \| ) 
\leq C' \| u \|_{L^{2}}
\]
for any $u \in L^{2}(S)$. 
Here, 
$\| \cdot \|_{H^{1}}$ is the Sobolev first norm on a compact set $\mathrm{Supp}(g) \times N$. 
By the Rellich lemma, we have $g (D_{s} + (1-s)\epsilon )^{-1} \sim 0$. 
Thus we also have 
$(D_{s} + (1-s)\epsilon )^{-1}g = (\bar{g}(D_{s} + (1-s)\epsilon )^{-1})^{\ast} \sim 0$. 

Second, we show 
$[ \varphi ,  (D_{s} + (1-s)\epsilon )^{-1}] \sim 0$ for any 
$\varphi \in C^{\infty}(\mathbb{R})$ satisfying $\varphi = \Pi$ 
on the complement of a compact set in $M$. 
Since $\varphi$ commutes with a operator on $N$, 
we have 
\[ [ \varphi ,  (D_{s} + (1-s)\epsilon )^{-1}]
	 = (D_{s} + (1-s)\epsilon )^{-1} 
	\begin{bmatrix} 0 & \varphi' \\ -\varphi' & 0 \end{bmatrix} (D_{s} + (1-s)\epsilon )^{-1}
	\sim 0.  \]

By a similar proof in the proof of Proposition \ref{prp:loccpt} (ii), 
we have $[ \Pi ,  (D_{s} + (1-s)\epsilon )^{-1}] \sim 0$. 
This proves $\Pi v_{\phi , s} \sim v_{\phi , s} \Pi$ and 
thus $
\Pi u_{\phi, s} \Pi : \Pi (L^{2}(S)) \to \Pi (L^{2}(S))
$
is a Fredholm operator. 

%
\end{proof}

Due to Propositions \ref{prp:homotopy} and \ref{prp:inRoe},  
\[ 
\ind (\Pi \varrho(\phi ) \Pi : \Pi (L^{2}(S)) \to \Pi (L^{2}(S)))
\]
is equal to $\ind (\Pi u_{\phi , 1} \Pi )$. 
Let $H : L^{2}(\mathbb{R}) \to L^{2}(\mathbb{R})$ be the Hilbert transformation:  
\[
Hf(t) = -\frac{i}{\pi}\mathrm{p.v.}\int_{\mathbb{R}}\frac{f(y)}{t-y}dy.
\]
Then the eigenvalues of $H$ are only $1$ and $-1$ by $H^{2} = 1$ and $H \neq \pm 1$. 
Let $\mathscr{H}_{-}$ be the $(-1)$-eigenspace of $H$ and 
$\hat{P} : L^{2}(\mathbb{R}) \to \mathscr{H}_{-}$ the projection to $\mathscr{H}_{-}$. 

\begin{prp}
Set $\mathscr{T}_{\phi} = (-it+F)^{-1}\phi (-it+F)$. 
Then $\hat{P}\mathscr{T}_{\phi}\hat{P}^{\ast}$ is a Fredholm operator and 
one has 
\[
\ind (\Pi \varrho(\phi) \Pi : \Pi (L^{2}(S)) \to \Pi (L^{2}(S)))
	= \ind (\hat{P}\mathscr{T}_{\phi}\hat{P}^{\ast} 
	: X \to X), 
\]
where $X = \mathscr{H}_{-}\otimes L^{2}(S_{N})$. 
\end{prp}

\begin{proof}
Due to Propositions \ref{prp:homotopy} and \ref{prp:inRoe}, we have 
\begin{equation*}
\ind (\Pi \varrho (\phi) \Pi : \Pi (L^{2}(S)) \to \Pi (L^{2}(S))) 
= \ind (\Pi u_{\phi , 1} \Pi : \Pi (L^{2}(S)) \to \Pi (L^{2}(S))). 
\end{equation*}
Because of 
\begin{equation*}
u_{\phi , 1} 
= 
\begin{bmatrix}
1 & 0 \\ 0 & (\der / \der t + F)^{-1}\phi (\der / \der t + F) 
\end{bmatrix}, 
\end{equation*}
the quantity $\ind (\Pi \varrho (\phi) \Pi : \Pi (L^{2}(S)) \to \Pi (L^{2}(S)))$ 
equals 
\begin{equation*}
\ind (\Pi (\der / \der t + F)^{-1}\phi (\der / \der t + F) \Pi 
	: \Pi (L^{2}(\mathbb{R})) \otimes L^{2}(S_{N}) \to \Pi (L^{2}(\mathbb{R})) \otimes L^{2}(S_{N})). 
\end{equation*}

Let $\mathscr{F} : L^{2}(\mathbb{R}) \to L^{2}(\mathbb{R})$ be the Fourier transformation: 
\[
\mathscr{F}[f](\xi ) = \int_{\mathbb{R}}e^{-ix\xi}f(x)dx.
\]
Then, we have $\mathscr{F}^{-1} \Pi \mathscr{F} = (1-H)/2 = \hat{P}$ 
and $\mathscr{F}^{-1} \der / \der t \mathscr{F} = -it$. 
This implies  
\begin{equation*}
\ind (\Pi \varrho (\phi) \Pi : \Pi (L^{2}(S)) \to \Pi (L^{2}(S))) 
= \ind (\hat{P}\mathscr{T}_{\phi}\hat{P}^{\ast} 
	: X \to X).
\end{equation*}
\end{proof}

Thus it suffices to calculate $\ind (\hat{P}\mathscr{T}_{\phi}\hat{P}^{\ast} )$ 
in order to prove the main theorem. 
For this purpose, we use eigenfunctions of the Hilbert transformation. 

\begin{lem}
\cite[Theorem 1]{MR1277773} 
Define $a_{n} \in L^{2}(\mathbb{R})$ by 
\[ a_{n}(t) = \frac{(t-i)^{n}}{(t+i)^{n+1}} \]
for all $n \in \mathbb{Z}$. 
Then $\{a_{n}/\sqrt{\pi} \}$ is an orthonormal basis of $L^{2}(\mathbb{R})$ 
and 
\begin{equation*}
Ha_{n} = 
\begin{cases}
a_{n} & \text{if~} n < 0 \\
-a_{n} & \text{if~} n \geq 0
\end{cases}.
\end{equation*}
This implies $\mathscr{H}_{-} = \mathrm{Span}_{\mathbb{C}}\{ a_{n} \}_{n \geq 0}$. 
\end{lem}

\begin{prp}
One has $\ind (\hat{P}\mathscr{T}_{\phi}\hat{P}^{\ast}) = \ind (T_{\phi})$. 
Therefore Theorem \ref{thm} in the case when $M = \mathbb{R} \times N$ holds. 
\end{prp}

\begin{proof}
Set $X_{0} = \mathbb{C}\{ a_{0} \} \otimes H_{+}$ and 
$X_{1} = (\mathrm{Span}_{\mathbb{C}} \{ a_{n}\}_{n \geq 1} \otimes H_{+}) 
	\oplus (\mathscr{H}_{-} \otimes H_{-})$. 
We note that we have $X_{0} \oplus X_{1} = \mathscr{H}_{-}\otimes L^{2}(S_{N}) = X$. 
Let $p : \mathscr{H}_{-} \to \mathbb{C}\{ a_{0} \}$ be the projection to $\mathbb{C}\{ a_{0} \}$. 
Then $p_{0} = p \otimes P : X \to X_{0}$ is the projection to $X_{0}$ and 
$p_{1} = \mathrm{id}_{X} - p_{0} : X \to X_{1}$ is the projection to $X_{1}$. 

By the decomposition of $L^{2}(S_{N}) = H_{+} \oplus H_{-}$, we have 
\begin{equation*}
\mathscr{T}_{\phi} 
= 
\begin{bmatrix}
\mathrm{id}_{L^{2}(\mathbb{R})} \otimes P\phi P^{\ast} & \frac{t-i}{t+i} \otimes P\phi (1-P)^{\ast} \\
\frac{t+i}{t-i} \otimes (1-P)\phi P^{\ast} & \mathrm{id}_{L^{2}(\mathbb{R})} \otimes (1-P)\phi (1-P)^{\ast}
\end{bmatrix}.
\end{equation*}
So we obtain 
$
\hat{P}\mathscr{T}_{\phi}\hat{P}^{\ast}p_{0}^{\ast} 
= 
p^{\ast} \otimes P\phi P^{\ast}
= 
\mathrm{id}_{\mathbb{C}\{ a_{0} \}} \otimes T_{\phi} 
$
and 
\[
\mathscr{T}_{\phi}\hat{P}^{\ast}p_{1}^{\ast}
= 
\begin{bmatrix}
(\hat{P} - p)^{\ast} \otimes P\phi P^{\ast} & \frac{t-i}{t+i}\hat{P}^{\ast} \otimes P\phi (1-P)^{\ast} \\
\frac{t+i}{t-i}(\hat{P} - p)^{\ast} \otimes (1-P)\phi P^{\ast} 
	& \hat{P}^{\ast} \otimes (1-P)\phi (1-P)^{\ast}
\end{bmatrix}. 
\]
This implies  
$\mathrm{Image}(\hat{P}\mathscr{T}_{\phi}\hat{P}^{\ast}p_{0}^{\ast}) \subset X_{0}$, 
$\mathrm{Image}(\mathscr{T}_{\phi}\hat{P}^{\ast}p_{1}^{\ast}) \subset X_{1}$
and 
\[
(\hat{P}\mathscr{T}_{\phi^{-1}}\hat{P}^{\ast}p_{1}^{\ast})(\hat{P}\mathscr{T}_{\phi}\hat{P}^{\ast}p_{1}^{\ast})
= 
\hat{P} \mathscr{T}_{\phi^{-1}}\mathscr{T}_{\phi}\hat{P}^{\ast}p_{1}^{\ast} 
= 
\mathrm{id}_{X_{1}}. 
\]
So $\hat{P}\mathscr{T}_{\phi}\hat{P}^{\ast}$ forms a direct sum of 
an invertible part $\hat{P}\mathscr{T}_{\phi}\hat{P}^{\ast}p_{1}^{\ast}$ and 
another part $\hat{P}\mathscr{T}_{\phi}\hat{P}^{\ast}p_{0}^{\ast}$: 
\[
\hat{P}\mathscr{T}_{\phi}\hat{P}^{\ast}
= 
\begin{bmatrix}
\hat{P}\mathscr{T}_{\phi}\hat{P}^{\ast}p_{0}^{\ast} & 0 \\
0 & \hat{P}\mathscr{T}_{\phi}\hat{P}^{\ast}p_{1}^{\ast} 
\end{bmatrix} \text{~on~} X_{0} \oplus X_{1}. 
\]
This proves
$\ind (\hat{P}\mathscr{T}_{\phi}\hat{P}^{\ast}) 
	= \ind (\hat{P}\mathscr{T}_{\phi}\hat{P}^{\ast}p_{0}^{\ast}) 
	= \ind ( T_{\phi})$. 
\end{proof}

We note that we also get 
\[
\ind \left(\Pi u_{\chi, \phi} \Pi\right) 
= 
\ind (T_{\phi}). 
\]

%
\section{The general case}
\label{sec:gen}
%

In this section we reduce the proof for the general partitioned manifold to 
that of $\mathbb{R} \times N$. 
Our argument is similar to Higson's argument in 
\cite{MR1113688}. 
By above sections, 
it suffices to show the case when $\phi \in GL_{l}(\mathscr{W}(M))$. 
Firstly, we shall show a cobordism invariance. 
See also 
\cite[Lemma 1.4]{MR1113688}. 

\begin{lem}
\label{cobor}
Let $(M^{+}, M^{-}, N)$ and $(M^{+}{}', M^{-}{}' , N')$ be two partitions of $M$. 
Assume that these two partitions are cobordant, that is, 
symmetric differences $M^{\pm} \triangle M^{\mp}{}'$ are compact. 
Let $\Pi$ and $\Pi'$ be the characteristic function of $M^{+}$ and $M^{+}{}'$, respectively. 
Take $\phi \in GL_{l}(\mathscr{W}(M))$. 
Then one has $\ind (\Pi u_{\chi , \phi} \Pi) = \ind (\Pi' u_{\chi , \phi} \Pi')$ 
and $\ind (\Pi \varrho (\phi) \Pi) = \ind (\Pi' \varrho (\phi) \Pi')$. 
\end{lem}

\begin{proof}
It suffices to show the case when $l=1$. 
Since we have $[\phi , \Pi ] = 0$ and $[u_{\chi , \phi} , \Pi ] \sim 0$, 
we obtain 
\begin{align*}
 & \ind (\Pi u_{\chi , \phi} \Pi : \Pi (L^{2}(S)) \to \Pi (L^{2}(S))) \\
=& \ind \left((1- \Pi ) 
	\begin{bmatrix} 1 & 0 \\ 0 & \phi \end{bmatrix} + \Pi u_{\chi , \phi} : L^{2}(S) \to L^{2}(S) \right) \\
=& \ind \left(
	\begin{bmatrix} 1 & 0 \\ 0 & \phi \end{bmatrix} + \Pi v_{\chi , \phi} : L^{2}(S) \to L^{2}(S) \right).  
\end{align*}
Therefore, it suffices to show 
$\Pi v_{\chi , \phi} \sim \Pi' v_{\chi , \phi}$. 
Now, since $M^{\pm} \triangle M^{\mp}{}'$ are compact, 
there exists $f \in C_{0}(M)$ such that $\Pi - \Pi' = (\Pi - \Pi')f$. 
So we obtain $\Pi v_{\chi , \phi} - \Pi' v_{\chi , \phi} = (\Pi - \Pi') fv_{\chi , \phi} \sim 0$. 
By the similar argument, we can prove 
$\ind (\Pi \varrho (\phi) \Pi) = \ind (\Pi' \varrho (\phi) \Pi')$. 
\end{proof}

Secondly, we shall prove an analogue of Higson's Lemma  
\cite[Lemma 3.1]{MR1113688}. 

\begin{lem}
\label{lem:Higson}
Let $M_{1}$ and $M_{2}$ be two partitioned manifolds and
$S_{j} \to M_{j}$ a Hermitian vector bundle.
Let $\Pi_{j}$ be the characteristic function of $M_{j}^{+}$. 
We assume that there exists an isometry $\gamma : M_{2}^{+} \to M_{1}^{+}$ which lifts
an isomorphism $\gamma^{\ast} : S_{1}|_{M_{1}^{+}} \to S_{2}|_{M_{2}^{+}}$. 
We denote the Hilbert space isometry defined by $\gamma^{\ast}$ 
by the same letter $\gamma^{\ast} : \Pi_{1}(L^{2}(S_{1})) \to \Pi_{2}(L^{2}(S_{2}))$. 
Take $u_{j} \in GL_{l}(C^{\ast}_{b}(M_{j}))$ such that 
$\gamma^{\ast}u_{1}\Pi_{1} \sim \Pi_{2}u_{2}\gamma^{\ast}$. 
Then one has $\ind (\Pi_{1}u_{1}\Pi_{1}) = \ind (\Pi_{2}u_{2}\Pi_{2})$. 

Similarly, if there exists an isometry $\gamma : M_{2}^{-} \to M_{1}^{-}$ which lifts 
an isomorphism $\gamma^{\ast} : S_{1}|_{M_{1}^{-}} \to S_{2}|_{M_{2}^{-}}$ and 
$\gamma^{\ast}u_{1}\Pi_{1} \sim \Pi_{2}u_{2}\gamma^{\ast}$, 
then one has $\ind (\Pi_{1}u_{1}\Pi_{1}) = \ind (\Pi_{2}u_{2}\Pi_{2})$.
\end{lem}

\begin{proof}
It suffices to show the case when $l=1$. 
Let $v : (1-\Pi_{1})(L^{2}(S_{1})) \to (1-\Pi_{2})(L^{2}(S_{2}))$ 
be any invertible operator. 
Then $V = \gamma^{\ast}\Pi_{1} + v(1-\Pi_{1}) : L^{2}(S_{1}) \to L^{2}(S_{2})$ is also invertible operator. 
Hence we obtain 
\begin{align*}
& \; V((1-\Pi_{1}) + \Pi_{1}u_{1}\Pi_{1}) - ((1-\Pi_{2}) + \Pi_{2}u_{2}\Pi_{2})V \\
=& \; \gamma^{\ast}\Pi_{1}u_{1}\Pi_{1} - \Pi_{2}u_{2}\Pi_{2}\gamma^{\ast} 
\sim  \gamma^{\ast}u_{1}\Pi_{1} - \Pi_{2}u_{2}\gamma^{\ast} \sim 0. 
\end{align*}
Therefore, we obtain $\ind (\Pi_{1}u_{1}\Pi_{1}) = \ind (\Pi_{2}u_{2}\Pi_{2})$ 
since $V$ is an invertible operator and one has 
$\ind (\Pi_{j}u_{j}\Pi_{j}) = \ind ((1-\Pi_{j}) + \Pi_{j}u_{j}\Pi_{j})$ 
for $j=1,2$. 
\end{proof}

Applying Lemma \ref{lem:Higson}, we prove the following: 

\begin{cor}
\label{ourH}
Let $M_{1}$ and $M_{2}$ be two partitioned manifolds. 
Let $S_{j} \to M_{j}$ be a graded Clifford bundle 
with the grading $\epsilon_{j}$, 
and denote by $D_{j}$ the graded Dirac operator of $S_{j}$. 
We assume that there exists an isometry $\gamma : M_{2}^{+} \to M_{1}^{+}$ 
which lifts isomorphism $\gamma^{\ast} : S_{1}|_{M_{1}^{+}} \to S_{2}|_{M_{2}^{+}}$ 
of graded Clifford structures. 
Moreover, we assume that 
$\phi_{j} \in GL_{l}(\mathscr{W}(M))$ satisfies 
$\phi_{1}(\gamma (x)) = \phi_{2}(x)$ for all $x \in M_{2}^{+}$. 
Then one has $\ind (\Pi_{1} u_{\chi , \phi_{1}} \Pi_{1}) = \ind (\Pi_{2}u_{\chi , \phi_{2}}\Pi_{2})$. 
\end{cor}

\begin{proof}
Fix small $R > 0$. It suffices to show 
$\gamma^{\ast}u_{\chi , \phi_{1}}\Pi_{1} \sim \Pi_{2} u_{\chi , \phi_{2}} \gamma^{\ast}$ 
the case when a chopping function $\chi \in C(\mathbb{R} ; [-1,1])$ 
satisfies $\mathrm{Supp}(\hat{\chi}) \subset (-R, R)$. 
Set $N_{2R} = \{ x \in M_{1}^{+} \,;\, d (x , N_{1}) \leq 2R \}$. 
Let $\varphi_{1}$ be a smooth function on $M_{1}$ such that 
$\mathrm{Supp} (\varphi_{1}) \subset M_{1}^{+} \setminus N_{2R}$ and 
assume that there exists 
a compact set $K \subset M_{1}$ such that $\varphi_{1} = \Pi_{1}$ on $M_{1} \setminus K$. 
Set $\varphi_{2}(x) = \varphi_{1}(\gamma (x))$ for all $x \in M_{2}^{+}$ 
and $\varphi_{2} = 0$ on $M_{2}^{-}$. 
Then we have $\gamma^{\ast}v_{\chi , \phi_{1}}\Pi_{1} \sim \gamma^{\ast}v_{\chi , \phi_{1}}\varphi_{1}$ and 
$\Pi_{2}v_{\chi , \phi_{2}}\gamma^{\ast} \sim \varphi_{2}v_{\chi , \phi_{2}}\gamma^{\ast}$. 
Thus, if we have 
$\gamma^{\ast}v_{\chi , \phi_{1}}\varphi_{1} \sim \varphi_{2}v_{\chi , \phi_{2}}\gamma^{\ast}$,  
then we obtain 
\begin{equation*}
\gamma^{\ast}u_{\chi , \phi_{1}}\Pi_{1}
\sim
\gamma^{\ast}v_{\chi , \phi_{1}}\varphi_{1} 
	+ \gamma^{\ast}\begin{bmatrix} 1 & 0 \\ 0 & \phi_{1} \end{bmatrix}\Pi_{1}
\sim
\varphi_{2}v_{\chi , \phi_{2}}\gamma^{\ast} 
	+ \Pi_{2}\begin{bmatrix} 1 & 0 \\ 0 & \phi_{2} \end{bmatrix}\gamma^{\ast}
\sim
\Pi_{2} u_{\chi , \phi_{2}} \gamma^{\ast}.
\end{equation*}
We shall show 
$\gamma^{\ast}v_{\chi , \phi_{1}}\varphi_{1} \sim \varphi_{2}v_{\chi , \phi_{2}}\gamma^{\ast}$. 
Now, we have $\gamma^{\ast}v_{\chi , \phi_{1}}\varphi_{1} = v_{\chi , \phi_{2}}\gamma^{\ast}\varphi_{1}$ 
since the propagation of $\chi (D)$ and $\eta (D)$ is less than $R$, respectively, and 
$\gamma^{\ast}D = D\gamma^{\ast}$ on $M^{+}$. 
Moreover, we have $[v_{\chi ,\phi_{2}} , \varphi_{2} ] \sim 0$ 
since $v_{\chi, \phi_{2}} \in M_{l}(C^{\ast}(M))$. 
Therefore, we obtain 
\begin{equation*}
\gamma^{\ast}v_{\chi , \phi_{1}}\varphi_{1} 
= v_{\chi , \phi_{2}}\gamma^{\ast}\varphi_{1} 
= v_{\chi , \phi_{2}}\varphi_{2}\gamma^{\ast} 
\sim \varphi_{2}v_{\chi , \phi_{2}}\gamma^{\ast}.  
\end{equation*}
\end{proof}

In order to prove Corollary \ref{cor:mcor}, 
we apply Lemma \ref{lem:Higson} as follows: 

\begin{cor}
\label{ourHcor}
We also assume as in Corollary \ref{ourH}. 
Then one has $\ind (\Pi_{1}\varrho (\phi_{1})\Pi_{1}) = \ind (\Pi_{2}\varrho (\phi_{2})\Pi_{2})$. 
\end{cor}

\begin{proof}
It suffices to show $\gamma^{\ast}\varrho (\phi_{1})\Pi_{1} \sim \Pi_{2} \varrho (\phi_{2}) \gamma^{\ast}$. 
Let $\varphi_{1}$ be a smooth function on $M_{1}$ such that 
$\mathrm{Supp} (\varphi_{1}) \subset M_{1}^{+}$ and 
assume that there exists 
a compact set $K \subset M_{1}$ such that $\varphi_{1} = \Pi_{1}$ on $M_{1} \setminus K$. 
Set $\varphi_{2}(x) = \varphi_{1}(\gamma (x))$ for all $x \in M_{2}^{+}$ 
and $\varphi_{2} = 0$ on $M_{2}^{-}$. 
Set 
$v_{\phi_{j}} = \varrho (\phi_{j}) - \begin{bmatrix} 1 & 0 \\ 0 & \phi_{j} \end{bmatrix}$. 
Then we have $\gamma^{\ast}v_{\phi_{1}}\Pi_{1} \sim \gamma^{\ast}v_{\phi_{1}}\varphi_{1}$ and 
$\Pi_{2}v_{\phi_{2}}\gamma^{\ast} \sim \varphi_{2}v_{\phi_{2}}\gamma^{\ast}$. 
Thus, if one has $\gamma^{\ast}v_{\phi_{1}}\varphi_{1} \sim \varphi_{2}v_{\phi_{2}}\gamma^{\ast}$, 
then we obtain 
\begin{equation*}
\gamma^{\ast}\varrho (\phi_{1})\Pi_{1}
\sim
\gamma^{\ast}v_{\phi_{1}}\varphi_{1} + \gamma^{\ast}\begin{bmatrix} 1 & 0 \\ 0 & \phi_{1} \end{bmatrix}\Pi_{1}
\sim
\varphi_{2}v_{\phi_{2}}\gamma^{\ast} + \Pi_{2}\begin{bmatrix} 1 & 0 \\ 0 & \phi_{2} \end{bmatrix}\gamma^{\ast}
\sim
\Pi_{2} \varrho (\phi_{2}) \gamma^{\ast}.
\end{equation*}
We shall show $\gamma^{\ast}v_{\phi_{1}}\varphi_{1} \sim \varphi_{2}v_{\phi_{2}}\gamma^{\ast}$. 
In fact, we obtain 
\allowdisplaybreaks
\begin{align*}
\gamma^{\ast}v_{\phi_{1}}\varphi_{1} - \varphi_{2}v_{\phi_{2}}\gamma^{\ast} 
& \sim \; \{\gamma^{\ast}\varphi_{1}(D_{1} + \epsilon_{1} )^{-1} 
		- (D_{2} + \epsilon_{2} )^{-1}\gamma^{\ast}\varphi_{1}\}
	\begin{bmatrix} \phi_{1} - 1 & -c(\mathrm{grad} (\phi_{1} ))^{-} \\ 0 & \phi_{1} - 1 \end{bmatrix} \\
& \sim \; (D_{2} + \epsilon_{2})^{-1}\gamma^{\ast}[D_{1}, \varphi_{1}](D_{1} + \epsilon_{1})^{-1}
	\begin{bmatrix} \phi_{1} - 1 & -c(\mathrm{grad} (\phi_{1} ))^{-} \\ 0 & \phi_{1} - 1 \end{bmatrix} 
\sim  \; 0 
\end{align*}
since $\mathrm{grad} (\varphi_{1})$ has a compact support and $[D_{1}, \varphi_{1}] = c(\mathrm{grad} (\varphi_{1}))$. 
Thus, we get $\gamma^{\ast}u_{\phi_{1}}\Pi_{1} \sim \Pi_{2} u_{\phi_{2}} \gamma^{\ast}$. 
Therefore, we obtain $\ind (\Pi_{1}u_{\phi_{1}}\Pi_{1}) = \ind (\Pi_{2}u_{\phi_{2}}\Pi_{2})$ 
by Lemma \ref{lem:Higson}. 
\end{proof}

\begin{proof}[Proof of Theorem \ref{thm}, the general case]
We assume $\phi \in GL_{l}(\mathscr{W}(M))$. 
Firstly, let $a \in C^{\infty}([-1,1];[-1,1])$ satisfies 
\[ a(t) = 
\begin{cases}
-1 & \text{if~} -1 \leq t \leq -3/4 \\
0 & \text{if~} -2/4 \leq t \leq 2/4 \\
1 & \text{if~} 3/4 \leq t \leq 1
\end{cases}. 
\]
Let $(-4\delta , 4\delta ) \times N$ be diffeomorphic to a tubular neighborhood of $N$ in $M$ satisfies 
\[
\sup_{(t,x),(s,y) \in [-3\delta , 3\delta ] \times N} |\phi (t,x) - \phi (s,y)| < \| \phi^{-1} \|^{-1}. 
\]
Set $\psi (t,x) = \phi (4\delta a(t) , x)$ on $(-4\delta , 4\delta ) \times N$ and 
$\psi = \phi$ on $M \setminus (-4\delta , 4\delta ) \times N$. 
Then we obtain $\psi \in GL_{l}(\mathscr{W}(M))$ and $\| \psi - \phi \| < \| \phi^{-1} \|^{-1}$. 
Thus a map 
$[0,1] \ni t \mapsto \psi_{t} = t\psi + (1-t)\phi \in GL_{l}(\mathscr{W}(M))$ 
is continuous with the uniform norm. 
Therefore it suffices to show the case when $\phi \in GL_{l}(\mathscr{W}(M))$ satisfying 
$\phi (t,x) = \phi (0,x)$ on $(-2\delta , 2\delta) \times N$. 
Due to Lemma \ref{cobor}, we may change a partition of $M$ 
to $(M^{+} \cup ([-\delta , 0] \times N), M^{-}\setminus ((-\delta, 0] \times N), \{ -\delta\} \times N)$ 
without changing $\ind (\Pi u_{\chi , \phi} \Pi)$. 
Then, due to Corollary \ref{ourH}, we may change $M^{+} \cup ([-\delta , 0] \times N)$ to 
$[-\delta , \infty ) \times N$ without changing $\ind (\Pi u_{\chi , \phi} \Pi)$. 
Here, $\phi$ is equal to $\phi (0,x)$ on $[-\delta , \infty ) \times N$ 
and the metric on $[0 , \infty) \times N$ is product. 
We denote this manifold by 
$M' = ([-\delta , \infty ) \times N) \cup (M^{-}\setminus ((-\delta, 0] \times N))$. 
$M'$ is partitioned by 
$([-\delta , \infty ) \times N, M^{-}\setminus ((-\delta, 0] \times N), \{ -\delta \} \times N)$. 
We apply a similar argument to $M'$, 
we may change $M'$ to a product $\mathbb{R} \times N$ without changing $\ind (\Pi u_{\chi , \phi} \Pi)$. 
Now we have changed $M$ to $\mathbb{R} \times N$. 
\end{proof}

\begin{proof}[Proof of Corollary \ref{cor:mcor}, the general case]
Similar. 
\end{proof}

\textit{Acknowledgments} - The author would like to thank 
Professor Hitoshi Moriyoshi for helpful comments.

\bibliographystyle{plain}
\bibliography{referrences_evenRoe}

\end{document}